\documentclass{birkjour}
 \newtheorem{theorem}{Theorem}[section]
 \newtheorem{corollary}[theorem]{Corollary}
 \newtheorem{lemma}[theorem]{Lemma}
 \newtheorem{proposition}[theorem]{Proposition}
 \theoremstyle{definition}
 \newtheorem{definition}[theorem]{Definition}
 \theoremstyle{remark}
 \newtheorem{remark}[theorem]{Remark}
 
 \numberwithin{equation}{section}

\usepackage{amsmath, amssymb, amsfonts}
\usepackage{mathrsfs}
 \newtheorem{problem}[theorem]{Problem}

\def\rr{{\mathbb R}}
\def\rn{{{\rr}^n}}
\def\D{\mathcal{D}}
\def\d{\mathrm{d}}

\def\H{\mathcal{H}}
\def\M{\mathcal{M}}

\def\sgn{\mathrm{sgn}}

\begin{document}

%
%
%
%
%
%
%
%
%

\title[The preduals of  Banach space valued Bourgain-Morrey  spaces]
 {The preduals of  Banach space valued Bourgain-Morrey  spaces}

\author[Tengfei Bai]{Tengfei Bai}

\address{%
Longkunnan 99,
 College of Mathematics and Statistics, Hainan Normal University, Haikou, Hainan 571158,
China
}

\email{202311070100007@hainnu.edu.cn}

\thanks{The corresponding author
	Jingshi Xu is supported by the National Natural Science Foundation of China (Grant No. 12161022) and the Science and Technology Project of Guangxi (Guike AD23023002).
	Pengfei Guo is supported by Hainan Provincial Natural Science Foundation of China (Grant No. 122RC652).}
\author{Pengfei Guo}
\address{Longkunnan 99,
	College of Mathematics and Statistics, Hainan Normal University, Haikou, Hainan 571158,
	China}
\email{050116@hainnu.edu.cn}

\author{Jingshi Xu}
\address{Jinji 1,
	School of Mathematics and Computing Science, Guilin University of Electronic Technology, Guilin 541004, China\\
Center for Applied Mathematics of Guangxi (GUET), Guilin 541004, China\\
Guangxi Colleges and Universities Key Laboratory of Data Analysis and Computation, Guilin 541004, China
}
\email{jingshixu@126.com}



\date{January 1, 2004}

\begin{abstract}
	Let $X$ be a Banach space  such that there exists a Banach space $^\ast X$ satisfying $  ( ^\ast X )^ \ast = X $. In this paper, we introduce $X$-valued  Bourgain-Morrey spaces. We show that $^\ast X$-valued block spaces  are the predual of $X$-valued  Bourgain-Morrey spaces. We obtain the completeness, denseness  and Fatou property of $^\ast X$-valued block spaces. We give a description of the dual of $X$-valued Bourgain-Morrey spaces and  conclude the reflexivity of these spaces. The boundedness of powered Hardy-Littlewood maximal operator in vector valued block spaces is obtained.
\end{abstract}

\maketitle
\section{Introduction}
Bourgain \cite{Bou91} introduced a special case of Bourgain-Morrey spaces to study the Stein-Tomas (Strichartz) estimate.
After then, many mathematicians began to study the Bourgain-Morrey spaces. For example,
in \cite{M16}, Masaki gave the block spaces which are the  preduals of Bourgain-Morrey spaces $M_{p}^{t,r}$ (Definition \ref{def:Bourgain Morrey space}) when $ 1\le p \le t \le r \le \infty $ and $t>1$. 
In \cite{HNSH23}, Hatano, Nogayama,  Sawano, and  Hakim researched the Bourgain-Morrey spaces from the viewpoints of harmonic analysis and functional analysis.   Specifically, they obtained that the block spaces are the  dual spaces of Bourgain-Morrey spaces  for $1<p<t<r<\infty$. After then, some function spaces extending Bourgain-Morrey spaces were established.
In \cite{ZSTYY23}, Zhao et al.  introduced Besov-Bourgain-Morrey spaces which connect Bourgain-Morrey spaces with amalgam-type spaces.  They obtained predual, dual spaces and complex interpolation  of these spaces. They also gave an equivalent norm with an integral expression  and  obtained the boundedness on these spaces
of the Hardy-Littlewood maximal operator, the fractional integral, and the Calder\'on-Zygmund operator. 

Immediately after \cite{ZSTYY23},
Hu, Li, and Yang  introduced  Triebel-Lizorkin-Bourgain-Morrey spaces which connect Bourgain-Morrey spaces and global Morrey spaces in \cite{HLY23}. 
They considered the embedding relations between Triebel-Lizorkin-Bourgain-Morrey spaces and Besov-Bourgain-Morrey spaces.  
They studied various fundamental real-variable properties of these spaces. They obtained the sharp boundedness of  the Hardy-Littlewood maximal operator, the Calder\'on-Zygmund operator, and the fractional integral on these spaces.

In \cite{H23}, Ho introduced and studied  grand Morrey spaces and grand Hardy-Morrey spaces on $\rn$. He defined the small block space and obtained a duality result between grand Morrey spaces and small block spaces. The boundedness of the Hardy-Littlewood maximal operator on small block spaces was also obtained. He extended the Rubio de Francia extrapolation theory to grand Morrey spaces and obtained the boundedness of the Calder\'on-Zygmund operators and the parametric Marcinkiewicz integrals on the grand Hardy-Morrey spaces.

Inspired by the generalized grand Morrey spaces and Besov-Bourgain-Morrey spaces, Zhang et al. introduced generalized grand Besov-Bourgain-Morrey spaces in \cite{ZYZ24}. They obtained predual spaces and the Gagliardo-Peetre  interpolation theorem,  extrapolation theorem.  The boundedness of Hardy-Littlewood maximal operator, the fractional integral and the Calder\'on-Zygmund operator on generalized grand Besov-Bourgain-Morrey spaces was also proved. 

The first author and the third author of this paper,  introduced the weighted Bourgain-Morrey-Besov-Triebel-Lizorkin spaces associated with operators over a  space of homogeneous type in \cite{BX25} and  obtained two sufficient conditions for precompact sets in matrix weighted Bourgain-Morrey spaces in \cite{BX24}.

The preduals  of Morrey spaces have been researched by many authors. 
In 1986, Zorko  \cite{Z86} proved that Morrey spaces have preduals.
In \cite{CH14}, Cheung and Ho proved the boundedness of the Hardy-Littlewood maximal operator on the block spaces with variable exponent, which are the preduals of the  Morrey space with variable exponent. 
In \cite{IST15}, Izumi et al. obtained  the Littlewood-Paley characterization of Morrey spaces and their preduals.
In \cite{ST15},  Sawano and Tanaka   obtained the Fatou property in the preduals of   Morrey  spaces. 
In \cite{RS16}, Rosenthal and Schmeisser showed the preduals of the vector-valued Morrey spaces. Furthermore,  the Calder\'on-Zygmund operators,   Hardy-Littlewood maximal operators, Fourier multipliers such as characteristic functions, smooth functions, strongly singular operators and Bochner-Riesz multipliers with critical indices are bounded in vector-valued Morrey-type spaces.
In \cite{H18}, Hakim  obtained the interpolation properties of the preduals  of generalized local Morrey spaces.
In  \cite{AS23}, Abe and Sawano  introduced the discrete Morrey spaces in terms of dyadic cubes and obtained Littlewood-Paley characterization of these spaces. They showed  that a predual of discrete Morrey spaces is the block space. As an application, they obtained the boundedness of martingale transforms, and the Riesz potential on discrete Morrey spaces. 
In \cite{W23}, Wei showed the preduals of Herz-Morrey spaces and the boundedness of Hardy-Littlewood maximal function. Using this, the extrapolation theory for Herz-Morrey spaces is established.
In \cite{W24}, Wei  obtained the boundedness of the strong maximal operator on weighted product block spaces which are  the preduals of weighted product Morrey spaces and extended the extrapolation theory on weighted product Morrey spaces.
In \cite{N24},  Nogayama obtained the Littlewood-Paley, heat semigroup and wavelet  characterizations for mixed Morrey spaces and its predual spaces. 
In \cite{DDDN24}, Dung, Dao, Duong and Nghia showed that the generalized block spaces are the preduals of certain generalized Morrey-Lorentz spaces. Moreover, they gave the mapping properties  for various classes of operators in  the generalized block spaces and generalized Morrey-Lorentz spaces, concerning Calder\'on-Zygmund type operator, commutators of Calder\'on-Zygmund type.
A comprehensive treatment of the preduals of Morrey  spaces can be seen in the monographs \cite[Chapter 5]{Ada15}, \cite[Chapter 9]{SFH20}.

The above literature considered the scalar ($\mathbb R$ or $\mathbb C$). Since $\mathbb R$ is a special Banach space, it is meaningful to research the Banach space valued Bourgain-Morrey  spaces and its predual. 
We also want to  research the boundedness of the powered Hardy-Littlewood maximal function on vector valued block spaces which will be applied in establishment of matrix weighted Bourgain-Morrey Triebel-Lizorkin spaces.

An outline of this article is as follows:
In Section \ref{pre}, we give some notations and concepts, including dyadic cubes, Hardy-Littlewood maximal function,  Banach space valued Bourgain-Morrey spaces $M_p^{t,r} (X)$, Banach space valued the block spaces $\mathcal{H}_{p'}^{t',r'} (^\ast X)$ ($^\ast X$ is the predual of $X$).
We show $L_c^\infty (X)$  is dense in $M_{p}^{t,r} (X)$ for $1 \le p<t<r<\infty$.
One of the main result of the paper that $\mathcal{H}_{p'}^{t',r'} ( ^\ast X)$ is a predual of $M_p^{t,r} (X)$, is obtained in Section  \ref{sec predual}.
In Section \ref{property block}, we show the completeness, denseness and Fatou property of block spaces $\mathcal{H}_{p'}^{t',r'} (^\ast X)$. The lattice  property of $\mathcal{H}_{p'}^{t',r'} (\ell^{q'})$ is obtained in this section.
In the last of this section, we obtain the dual of  $  M_p^{t,r} (X) $ is $ \H_{p'}^{t',r'}(  X ^\ast)  $ under the assumption that $X$  is reflexive and $1<p<t<r<\infty$.
In Section \ref{HL block},  the boundedness of powered Hardy-Littlewood maximal operator on vector valued block spaces  $\mathcal{H}_{p'}^{t',r'} (\ell^{q'})$ is obtained.

Throughout this paper, we let $c, C$ denote constants that are independent of the main parameters involved but whose value may differ from line to line.
By $A\lesssim B$ we mean that $A\leq CB$ with some positive constant $C$ independent of appropriate quantities. By $ A \approx B$, we mean that $A\lesssim B$ and $B\lesssim A$.
Let $\chi_{E}$ be the characteristic function of the set $E\subset\mathbb{R}^{n}$. Let $\mathbb N  =\{1,2,\ldots\} $ and let $\mathbb Z$ be all the integers. 
Let $\mathbb{N}_{0}:=\mathbb{N\cup}\{0\}$.
We let $ \sharp K $ be the cardinal number of the set $K$.
A vector valued functions $\vec f = \{f_i\}_{i=1}^\infty$ on $\rn $ is said measurable if each $f_i$ is a measurable function on $\rn$.
Let $\mathscr M (vec) $ be denotes the set of all measurable  vector valued functions on $\rn$.

\section{Preliminaries} \label{pre}
Throughout the paper, $X$  is a Banach space over the scalar field $\mathbb R$. The norm of an element $x \in X$  is denoted by $\|x\|_X$, or, if no confusion can arise, by $\|x\|$. The dual of $X$ is denoted by $X^\ast$. We use the notation $\langle x, x^* \rangle$ to denote the duality pairing of the elements $x\in X$ and $x^* \in X^\ast$.
Note that $ X^\ast$ is always a Banach space even if $X$  is a quasi-Banach space.

For $ 1 < p <\infty$, let $ p ' = p/(p-1)$. Let $1' =\infty  $ and $\infty ' = 1$.
For $j\in\mathbb{Z}$, $m\in\mathbb{Z}^{n}$, let $Q_{j,m}:=\prod_{i=1}^{n}[2^{-j}m_{i},2^{-j}(m_{i}+1))$.
For a cube $Q$, $\ell(Q)$ stands for the length of cube $Q$. We
denote by $\mathcal{D}$ the family of all dyadic cubes in $\mathbb{R}^{n}$,
while $\mathcal{D}_{j}$ is the set of all dyadic cubes with $\ell(Q)=2^{-j},j\in\mathbb{Z}$. Let $a Q , a >0$ be the  cube  concentric with $Q$, having  the side length of $a \ell (Q)$.
For any $R>0$, Let $ B(x,R) := \{y\in \rn : |x-y| <R \} $ be an open ball in $\rn $. Let $B_R : = B (0,R)$ for $R>0$.

Let $L^0 :=L^0 (\rn) $ be the  set of all measurable functions  on $\rn$.
Let $L_c^\infty : = L_c^\infty  (\rn) $ be the  set of all compactly supported bounded functions on $\rn$.
Let $ C_c^\infty : =  C_c^\infty  (\rn) $ be the  set of all  compactly supported smooth functions  on $\rn$.
For $0<p <\infty $, let $L_{\mathrm{loc}}^{p} : = L_{\mathrm{loc}}^{p} (\rn)$  be the set of locally $p$-integrable functions on $\rn$, that is, for each compact set $K \subset \rn$,  $ \int_K |f(x)|^p \d x <\infty  $.
For $0<p <\infty $, let $L_{\mathrm{loc}}^{p} (X) $  be the set of $X$-valued locally $p$-integrable functions on $\rn$, that is, for each compact set $K \subset \rn$,  $ \int_K \|f(x)\|_X^p \d x <\infty  $.

The Hardy-Littlewood maximal operator $\M$  is defined by 
\begin{equation*}
	\M f (x) : = \sup_{\operatorname{balls}  B \ni x} \frac{1}{|B|} \int_B |f(y)| \d y.
\end{equation*}
For $ 0< \eta <\infty$, we set the powered Hardy-Littlewood maximal operator $\M _\eta$ by 
\begin{equation*}
	\M _\eta f (x)  :=  \left( \M (f^\eta ) (x) \right)  ^{1/\eta}.
\end{equation*}
We define $\M _\eta  \vec f (x) : = \{ \M _\eta   f_i (x)   \}_{i=1}^\infty $ for a vector valued function $\vec f$. 

For the definition of Bourgain-Morrey
spaces, we use the notation in \cite{MS18} by Masaki and Segata.

\begin{definition}
	\label{def:Bourgain Morrey space} Let $0<p\le t<\infty$
	and $0<r\le\infty$. 
	The  Bourgain-Morrey space $M_{p}^{t,r}  $ consists of all  measurable functions $ f \in L_{\mathrm{loc}}^{p}$
	such that
	\[
	\| f \|_{M_{p}^{t,r}} :=\bigg\|\Big\{|Q_{v,m}|^{1/t-1/p}\Big(\int_{Q_{v,m}}|f(x)|^{p}\d x\Big)^{1/p}\Big\}_{v\in\mathbb{Z},m\in\mathbb{Z}^n}\bigg\|_{\ell^{r}}<\infty.
	\]
\end{definition}

\begin{definition}
	\label{def: X Bourgain Morrey space} Let $0<p\le t<\infty$
	and $0<r\le\infty$.  Let $X$  be a Banach space.
	The $X$-valued Bourgain-Morrey space $M_{p}^{t,r} (X) $ consists of all $X$-valued  functions $ f \in L_{\mathrm{loc}}^{p} (X)$
	such that
	\[
	\| f \|_{M_{p}^{t,r} (X)} :=  	\left\|  \|f \|_{X } \right\|_{M_{p}^{t,r}}  <\infty.
	\]
\end{definition}
Denote by $L^p (X) = M_{p}^{p,\infty} (X)$ the $X$-valued Lebesgue spaces.
\begin{lemma}[Theorem 2.10, \cite{HNSH23}] \label{not trivial}
	Let $0<p\le t<\infty$
	and $0<r\le\infty$. $M_{p}^{t,r} \neq \{ 0\}$ if and only if $ 0< p< t< r<\infty$  or $ 0 \le p \le t < r=\infty$.
\end{lemma}
\begin{remark}
	Let $ 0< p< t< r<\infty$  or $ 0 \le p \le t < r=\infty$.
	From Lemma \ref{not trivial}, we obtain $M_{p}^{t,r}(X)$  is not trivial.
	Indeed, let $f= \chi_{B(0,1)} \otimes h$ where  $h \in X$ with $\|h\|_{X} =1$. Then $f: \rn \to X$ and 
	\begin{equation*}
		0 < 	\| f \|_{M_{p}^{t,r} (X)} = \|h\|_{X} 	\|  \chi_{B(0,1)} \|_{M_{p}^{t,r} } <\infty.
	\end{equation*}
	
\end{remark}

For each $k \in \mathbb Z$, for a  $X$-valued measurable function $f$, define the average operator $E_k $  by
\begin{equation*}
	E_k (f) (x) = \sum_{Q \in \D_k} \frac{1}{|Q|} \int_Q f (y) \d y \chi_Q (x), \quad x \in \rn.
\end{equation*}

\begin{theorem} \label{E_k f to f}
	Let $X$ be a Banach space. Let $1\le p<t<r<\infty$. Then
	\begin{equation*}
		\lim_{k\to \infty} \| f - E_k (f) \|_{M_{p}^{t,r} (X) } =0.
	\end{equation*}
\end{theorem}
\begin{proof}
	We use the idea from \cite[Theorem 2.20]{HNSH23}.
	Let $f\in M_{p}^{t,r} (X)$ and $\epsilon>0 $ be fixed. Since $f \in M_{p}^{t,r} (X)$, there exists $J \in \mathbb N$ large enough such that
	\begin{align} \label{J epsilon}
		\nonumber
	&	\left(\sum_{j\in \mathbb Z}  \sum_{m \in \mathbb Z^n}  ( \chi_{\mathbb Z \backslash [-J,J]} (j)  + \chi_{\mathbb Z^n \backslash B_J  }  (m) )  |Q_{j,m}|^{r/t-r/p} \left( \int_{Q_{j,m} }  \|f(x)\|_X ^p \d x \right)^{r/p} \right)^{1/r} \\
	& <\epsilon.
	\end{align}
	Fix $k \in \mathbb N \cap (J,\infty)$. Set
	\begin{align*}
		I & := \left\| \left\{       |Q_{j,m}|^{1/t-1/p} \left( \int_{Q_{j,m}} \| f (y) -E_k (f) (y) \|_X^p \d y \right)^{1/p}  \right\}_{j \in [-J,J], m\in \mathbb Z^n}  \right\|_{\ell^r} , \\
		II & := \left\| \left\{       |Q_{j,m}|^{1/t-1/p} \left( \int_{Q_{j,m}} \| f (y)  \|_X^p \d y \right)^{1/p}  \right\}_{j \in \mathbb Z \backslash [-J,J], m\in \mathbb Z^n}  \right\|_{\ell^r} ,\\
		III & := \left\| \left\{       |Q_{j,m}|^{1/t-1/p} \left( \int_{Q_{j,m}} \| E_k  f (y)  \|_X^p \d y \right)^{1/p}  \right\}_{j \in ( (-\infty, -J) \cup (J,k]  )\cap \mathbb Z, m\in \mathbb Z^n}  \right\|_{\ell^r}, \\
		IV & := \left\| \left\{       |Q_{j,m}|^{1/t-1/p} \left( \int_{Q_{j,m}} \| E_k  f (y)  \|_X^p \d y \right)^{1/p}  \right\}_{j \in  \mathbb Z \backslash (-\infty,k], m\in \mathbb Z^n}  \right\|_{\ell^r}.
	\end{align*}
	Then we decompose
	\begin{equation*}
		\| f-  E_k (f) \|_{M_{p}^{t,r} (X) }  \le I +II+ III+ IV.
	\end{equation*}
	For $I$, by Minkowski's inequality, we see
	\begin{align*}
		I & \le \left\| \left\{       |Q_{j,m}|^{1/t-1/p} \left( \int_{Q_{j,m}} \| f (y) -E_k (f) (y) \|_X^p \d y \right)^{1/p}  \right\}_{j \in [-J,J], m\in \mathbb Z^n \cap B_J}  \right\|_{\ell^r} \\
		& \quad +\epsilon.
	\end{align*}
	Since $E_k (f)$  converges back to $f$ in $L_{\operatorname{loc}} ^p (X)$ and $L^p (X)$ (\cite[Theorem 3.3.2]{HVVW16}), we obtain
	\begin{equation*}
		I  < 2\epsilon
	\end{equation*} 
	as long as $k$ large enough.
	
	We move on $II$.  Simply use (\ref{J epsilon}) to conclude $II< \epsilon$.
	
	Next, we estimate $III$. For each $j \in \mathbb Z\cap (-\infty,k]$, by the disjointness of dyadic cubes,   general Minkowski's inequality, H\"older's inequality, we have
	\begin{align*}
		\left( \int_{Q_{j,m}} \| E_k  f (y)  \|_X^p \d y \right)^{1/p} & = \left( \sum_{Q \in \D_k, Q \subset Q_{j,m} } \int_{Q} \left\| |Q|^{-1} \int_Q f (z) \d z  \right\|_X^p \d y \right)^{1/p} \\
		& = \left( \sum_{Q \in \D_k, Q \subset Q_{j,m} } \left\| |Q|^{-1} \int_Q f (z) \d z  \right\|_X^p |Q| \right)^{1/p} \\
		& \le  \left( \sum_{Q \in \D_k, Q \subset Q_{j,m} }  |Q|^{-1}  \int_Q \|f (z)\|_X^p \d z  |Q| \right)^{1/p} \\
		& =  \left(  \int_{ Q_{j,m} } \|f (z)\|_X^p \d z   \right)^{1/p} .
	\end{align*}
	Taking the $\ell^r$-norm over $j \in ( (-\infty, -J) \cup (J,k]  )\cap \mathbb Z, m\in \mathbb Z^n $, we obtain
	\begin{align*}
		III \le \left(   \sum_{ j \in ( (-\infty, -J) \cup (J,k]  )\cap \mathbb Z}   \sum_{ m\in \mathbb Z^n }  |Q_{j,m}|^{r/t-r/p}  \left(  \int_{ Q_{j,m} } \|f (z)\|_X^p \d z  \right)^{r/p} \right)^{1/r} <\epsilon.
	\end{align*}
	Finally, we estimate $IV$. Let $j \in \mathbb Z \backslash (-\infty,k]$. Since $Q_{j,m} \subset Q$ for $Q\in \D_k$  as long as $Q_{j,m} \cap Q \neq\emptyset$, we have
	\begin{align*}
		| IV |^r & \le  \sum_{j=k+1}^\infty \sum_{m \in \mathbb Z^n}  |Q_{j,m}|^{r/t-r/p} \left( \int_{Q_{j,m}} \| E_k  f (y)  \|_X^p \d y \right)^{r/p} \\
		& = \sum_{j=k+1}^\infty  \sum_{m \in \mathbb Z^n,  Q\in \D_k, Q_{j,m} \subset Q}  |Q_{j,m}|^{r/t} \left\| \frac{1}{|Q|} \int_Q f  (z) \d z \right\|_X ^r.
	\end{align*}
	Note that $   |Q_{j,m}| = 2^{ -(j-k)n } |Q|$. Since $t<r<\infty$,  by  general Minkowski's inequality, H\"older's inequality,
	we obtain
	\begin{align*}
		| IV |^r & \le   \sum_{j=k+1}^\infty  \sum_{m \in \mathbb Z^n,  Q\in \D_k, Q_{j,m} \subset Q}  2^{ -(j-k)n r/t } |Q|^{r/t} \left\| \frac{1}{|Q|} \int_Q f  (z) \d z \right\|_X  ^r \\
		& \le  \sum_{j=k+1}^\infty  \sum_{m \in \mathbb Z^n,  Q\in \D_k, Q_{j,m} \subset Q}  2^{ -(j-k)n r/t } |Q|^{r/t-r/p} \left(  \int_Q \|f  (z) \|^p \d z  \right) ^{r/p} \\
		& \le \sum_{j=k+1}^\infty  2^{ -(j-k)n (1- r/t) }\sum_{  Q\in \D_k}|Q|^{r/t-r/p} \left(  \int_Q \|f  (z) \|^p \d z  \right) ^{r/p} \\
		& \lesssim \sum_{  Q\in \D_k}|Q|^{r/t-r/p} \left(  \int_Q \|f  (z) \|^p \d z  \right) ^{r/p} .
	\end{align*}
	Since $J \le k$, from (\ref{J epsilon}), we conclude $IV \lesssim \epsilon$.
	
	Combining the estimates $I -IV$, we get 
	\begin{equation*}
		\limsup_{k\to \infty} \| f-  E_k (f) \|_{M_{p}^{t,r} (X) }   \lesssim \epsilon.
	\end{equation*} 
	Since $\epsilon>0$ is arbitrary, we  obtain the desired result.
\end{proof}

As a further corollary, the set $L_c^\infty (X)$ of all compactly supported $X$-valued, bounded functions  is dense in $M_{p}^{t,r} (X)$.
\begin{corollary}
	Let $X$ be a Banach space.  Let $1 \le p<t<r<\infty$. Then $L_c^\infty (X)$  is dense in $M_{p}^{t,r} (X)$.
\end{corollary}
\begin{proof}
	Fix $\epsilon>0$. By Theorem \ref{E_k f to f}, there exists $k$ large enough such that 
	\begin{equation} \label{f -Ek f < eps}
		\| f - E_k (f) \|_{M_{p}^{t,r} (X) } < \epsilon.
	\end{equation}
	Note that $E_k (f)\chi_{B_R}  \in L_c^\infty (X) $  for all $R>0$.
	Fix a large number $k\in \mathbb N$ such that (\ref{f -Ek f < eps}) holds. By the disjointness of the family $\{ Q_{k,m} \}_{m\in \mathbb Z^n}$, 
	\begin{align*}
		&	\| E_k (f) - E_k (f)\chi_{B_R} \|_{ M_{p}^{t,r} (X)  }  \\
		& = \left( \sum_{j\in \mathbb Z, m\in \mathbb Z^n} |Q_{j,m}|^{r/t-r/p}   \left( \int_{ Q_{j,m}  } \|E_k (f) - E_k (f)\chi_{B_R} \|_X ^p \d x      \right)^{r/p}   \right)^{1/r}  \\
		& \le  \left( \sum_{j\in \mathbb Z}  \sum_{ m\in \mathbb Z^n, Q_{j,m} \cap B_R^c \neq \emptyset} |Q_{j,m}|^{r/t-r/p}   \left( \int_{ Q_{j,m}  } \|E_k (f) \|_X ^p \d x      \right)^{r/p}   \right)^{1/r} .
	\end{align*}
	Note that $E_k (f) \in M_{p}^{t,r} (X)$ according to Theorem \ref{E_k f to f}. By the Lebesgue convergence theorem  (see \cite[Proposition 1.2.5]{HVVW16} for Banach space valued dominated convergence theorem) and the fact $f\in M_{p}^{t,r} (X)$, we obtain
	\begin{equation*}
		\| E_k (f) - E_k (f)\chi_{B_R} \|_{ M_{p}^{t,r} (X)  }   <\epsilon
	\end{equation*}
	as long as $R$ large enough. 
	Then by Minkowski's inequality, we have
	\begin{equation*}
		\| f - E_k (f)\chi_{B_R} \|_{M_{p}^{t,r} (X) } < 2\epsilon.
	\end{equation*}
	Since $\epsilon>0$ is arbitrary, the proof is complete.
\end{proof}

\section{Predual  of $M_{p}^{t,r} (X)$} \label{sec predual}
In this section, we first establish the predual space of $X$-valued Bourgain-Morrey spaces. 
We further suppose that $X$ is the dual of  a Banach space. That is there exists a Banach space $^\ast X$ such that $  ( ^\ast X )^ \ast = X $. 
Then we recall the block spaces and the scalar block spaces can been seen in \cite[Definition 2.15]{M16}.
\begin{definition}
	Let $1\le p\le t\le\infty$ and $X$ be a Banach space such that $  ( ^\ast X )^ \ast = X $. A $^\ast X$-valued function $b$ is a $(p',t',{^\ast X})$-block
	if there exists a cube $Q$ which supports $b$ such that
	\[
	\left\| \|b\|_{^\ast X} \right\|_{L^{p'}}\le |Q|^{1/t-1/p} = |Q|^{1/p' -1/t'}.
	\]
	If we need to indicate $Q$, we  say that $b$ is a $(p',t', {^\ast X})$-block supported on $Q$.
\end{definition}

\begin{definition}\label{def block space}
	Let $1\le p\le t\le r\le\infty$ and $X$ be a Banach space such that $  ( ^\ast X )^ \ast = X $. The function space $\mathcal{H}_{p'}^{t',r'} (^\ast X)$
	is the set of  all $^\ast X$-valued measurable functions $f$ such that $f$ is realized
	as the sum
	\begin{equation}\label{eq:block f}
		f = \sum_{(j,k)\in\mathbb{Z}^{n+1}}\lambda_{j,k}b_{j,k}
	\end{equation}
	with some $\lambda=\{\lambda_{j,k}\}_{(j,k)\in\mathbb{Z}^{n+1}}\in\ell^{r'}(\mathbb{Z}^{n+1})$
	and $b_{j,k}$ is a  $(p',t', {^\ast X})$-block supported on $Q_{j,k}$ where (\ref{eq:block f}) converges in $^\ast X$ almost everywhere on $\rn$. The norm of $\mathcal{H}_{p'}^{t',r'} (^\ast X)$
	is defined by
	\[
	\|f\|_{\mathcal{H}_{p'}^{t',r'} (^\ast X)} :=\inf_{\lambda}\|\lambda\|_{\ell^{r'}},
	\]
	where the infimum is taken over all admissible sequence $\lambda$
	such that (\ref{eq:block f}) holds.
\end{definition}

\begin{remark}
	Let  $g_{j_0, k_0}$ be a $(p',t',{^\ast X})$-block supported on some $Q_{j_0, k_0}$. 
	Then $ \|g_{j_0, k_0} \|_{\mathcal{H}_{p'}^{t',r'} (^\ast X)}  \le 1$. Indeed, let 
	\begin{equation*}
		\lambda_{j,k} = \begin{cases}
			1, &  \operatorname{if} j = j_0, k = k_0; \\
			0 , & \operatorname{else}.
		\end{cases}
	\end{equation*} 
	Hence  $ \|g_{j_0, k_0} \|_{\mathcal{H}_{p'}^{t',r'} (^\ast X)}  \le  \left( \sum_{(j,k)\in   \mathbb Z ^{n+1}  } |	\lambda_{j,k} |^{r'}  \right)^{1/r'} =1.  $
\end{remark}

\begin{proposition}\label{block sum converge}
	Let $ 1 \le p < t < r \le \infty$. Let $X$ be a Banach space  such that $  ( ^\ast X )^ \ast = X $. Assume that $ \{\lambda_{j,k}\}_{(j,k)\in\mathbb{Z}^{n+1} }  \in \ell^{r'}$ and for each $(j,k)\in\mathbb{Z}^{n+1}$, $b_{j,k}$  is a  $(p',t', {^\ast X})$-block supported on $Q_{j,k}$. Then
	
	{\rm (i)} the summation (\ref{eq:block f}) converges  in $^\ast X $ almost everywhere   on $\rn$;
	
	{\rm (ii)} the summation (\ref{eq:block f}) converges   in $L^1_{\operatorname{loc}  }(^\ast X)$. 
\end{proposition}
\begin{proof}
	We use the idea from \cite[Propositon 3.4]{ZSTYY23}. 
	
	(i) We first claim that, for any given dyadic cube $Q\in \D$,
	\begin{equation} \label{eq almost every}
		\int_Q \sum_{(j,k)\in\mathbb{Z}^{n+1} } |\lambda_{j,k} |  \|b_{j,k} (x)\|_{^\ast X} \d x <\infty .
	\end{equation}
	Without loss of generality, we may assume that $Q := [0,1)^n $. Applying the Tonelli theorem, we obtain
	\begin{equation*}
		\int_Q \sum_{(j,k)\in\mathbb{Z}^{n+1} } |\lambda_{j,k} |  \|b_{j,k} (x)\|_{^\ast X} \d x = \sum_{(j,k)\in\mathbb{Z}^{n+1}, Q_{j,k} \cap Q \neq \emptyset }	 |\lambda_{j,k} | \int_Q   \|b_{j,k} (x)\|_{^\ast X} \d x .
	\end{equation*}
	By H\"older's inequality and support of $b_{j,k}$, we have
	\begin{align*}
		\int_Q    \|b_{j,k} (x)\|_{^\ast X} \d x  & \le |Q \cap Q_{j,k}|^{1/p}  \left\|  \|b_{j,k} \|_{^\ast X} \right\|_{L^{p'}} \\
		& \le \min\{ 1, 2^{-jn/p} \} 2^{-jn (1/t-1/p) }.
	\end{align*}
	Thus by $r >t $, we have
	\begin{align*}
	&	\sum_{j =0}^\infty \sum_{k \in \mathbb Z^n, Q_{j,k} \subset Q }	 |\lambda_{j,k} | \int_Q   \|b_{j,k} (x)\|_{^\ast X} \d x \\
		& \le \sum_{j =0}^\infty   2^{-jn/p} 2^{-jn (1/t-1/p) } \sum_{k \in \mathbb Z^n, Q_{j,k} \subset Q }	 |\lambda_{j,k} | \\
		& \le \sum_{j =0}^\infty    2^{-jn /t }  2^{jn/r} \left( \sum_{k \in \mathbb Z^n, Q_{j,k} \subset Q }	 |\lambda_{j,k} |^{r'} \right)^{1/r'} \\
		& \lesssim \| \{\lambda_{j,k} \}_{ (j,k)\in\mathbb{Z}^{n+1} } \|_{\ell^{r'}} < \infty .
	\end{align*}
	By $p < t$, we have
	\begin{align*}
		& \sum_{j =-\infty}^{-1} \sum_{k \in \mathbb Z^n, Q \subset Q_{j,k} }	 |\lambda_{j,k} | \int_Q   \|b_{j,k} (x)\|_{^\ast X} \d x \\
		& \le \sum_{j =-\infty}^{-1} 2^{-jn (1/t-1/p) }  \sum_{k \in \mathbb Z^n, Q \subset Q_{j,k} }	 |\lambda_{j,k} |  \\
		& = \sum_{j =-\infty}^{-1} 2^{-jn (1/t-1/p) }  \left( \sum_{k \in \mathbb Z^n, Q \subset Q_{j,k} }	 |\lambda_{j,k} |^{r'} \right)^{1/r'} \\
		& \lesssim \| \{\lambda_{j,k} \}_{ (j,k)\in\mathbb{Z}^{n+1} } \|_{\ell^{r'}} < \infty .
	\end{align*}
	Hence, from the above two estimates, we prove (\ref{eq almost every})  and the claim holds true.
	Using this claim, we find that, for any $Q\in \D$,
	\begin{equation*}
		\left\|  \int_Q \sum_{(j,k)\in\mathbb{Z}^{n+1} } \lambda_{j,k}   b_{j,k} (x)\d x \right\|_{^\ast X}  \le \int_Q \sum_{(j,k)\in\mathbb{Z}^{n+1} } |\lambda_{j,k} |  \|b_{j,k} (x)\|_{^\ast X} \d x <\infty.
	\end{equation*} 
	Thus, $ \sum_{(j,k)\in\mathbb{Z}^{n+1} } \lambda_{j,k}   b_{j,k} $ converges in  $^\ast X$ almost everywhere  on any $Q\in \D$. From the arbitrariness of $Q \in \D$, we infer that  $ \sum_{(j,k)\in\mathbb{Z}^{n+1} } \lambda_{j,k}   b_{j,k} $
	converges in  $^\ast X$ almost everywhere on $\rn$. Thus we prove (i).
	
	(ii) Applying (\ref{eq almost every}) and the Lebesgue dominated convergence theorem (see \cite[Proposition 1.2.5]{HVVW16} for Banach space valued dominated convergence theorem), we have that for any $Q\in \D$,
	\begin{equation*}
		\lim_{N \to \infty} \int_Q \Big\|  \sum_{ |j| < N , |k|_\infty < N}  \lambda_{j,k}   b_{j,k} (x) \Big\|_{^\ast X} \d x = \int_Q \Big\|  \sum_{ (j,k)\in\mathbb{Z}^{n+1} }  \lambda_{j,k}   b_{j,k} (x) \Big\|_{^\ast X} \d x,
	\end{equation*}
	which, together with (\ref{eq almost every}) again, further implies that the summation (\ref{eq:block f}) converges in $L^1_{\operatorname{loc}  }(^\ast X)$.
	Here $|k|_\infty = \max \{ k_1, \ldots, k_n\}$.
	Thus we prove (ii).
\end{proof}

\begin{remark}
	Based on Proposition \ref{block sum converge}, we find that, in Definition \ref{def block space}, instead of the requirement that   the summation (\ref{eq:block f}) converges in $^\ast X$ almost everywhere on $\rn$, if we require that (\ref{eq:block f}) converges in $L^1_{\operatorname{loc}  }(^\ast X)$, we then obtain the same block space.
\end{remark}

\begin{lemma}
	Let $X$ be a Banach space  such that $  ( ^\ast X )^ \ast = X $.
	Let $1< p = t <r =\infty $. Then  $\mathcal{H}_{p'}^{p',1} (^\ast X) = L^{p'} (^\ast X )$ with coincidence
	of norms.
\end{lemma}

\begin{proof}
	The scalar version can be seen in \cite[Proposition 339]{SFH20}. 
	
	If $f \in \mathcal{H}_{p'}^{p',1} (^\ast X )$. Then there exist a sequence $ \{ \lambda_{j,k}\} _{  (j,k)  \in \mathbb Z^{1+n} } \in  \ell^1  $ and a sequence $\{ b_{j,k} \}_{(j,k)  \in \mathbb Z^{1+n}  }$ of  $(p',p',{^\ast X}) $-blocks such that  $f = \sum_{ (j,k) \in \mathbb Z^{1+n}  }   \lambda_{j,k}  b_{j,k} $  and
	\begin{equation*}
		\sum_{  (j,k) \in \mathbb Z^{1+n} }    | \lambda_{j,k}|    \le  (1+\epsilon)	\| f\|_{\mathcal{H}_{p'}^{p',1} (^\ast X ) }.
	\end{equation*}
	Then 
	\begin{align*}
		\left\| \|f\|_{^\ast X}  \right\|_{ L^{p'}  } & =\left\| \left\|\sum_{ (j,k) \in \mathbb Z^{1+n}  }   \lambda_{j,k}  b_{j,k} \right\|_{^\ast X} \right\|_{L^{p'} } \\ & \le 	\sum_{  (j,k) \in \mathbb Z^{1+n} }    | \lambda_{j,k}|    \le  (1+\epsilon)	\| f\|_{\mathcal{H}_{p'}^{p',1} (^\ast X ) }.
	\end{align*}
	For the another direction, let $f \in  L^{p'} (^\ast X )$. Fix $\epsilon>0$. Then there is a decomposition 
	\begin{equation*}
		f = \chi_{Q_1} f + \sum_{j=2}^\infty  \chi_{Q_j \backslash Q_{j-1}} f, 
	\end{equation*}
	where $ \{ Q_j\}_{j=1}^\infty  $ is an increasing sequence of cubes centered at the origin satisfying 
	\begin{equation*}
		\| \chi_{\rn \backslash  Q_k } f \|_{  L^{p'} (^\ast X ) } \le \epsilon 2^{-k}
	\end{equation*}
	for all $k \in \mathbb N$. Then using this decomposition, we have
	\begin{equation*}
		\| f\|_{\mathcal{H}_{p'}^{t',r'} (^\ast X )  }  \le \| f\|_{ L^{p'} (^\ast X ) }  + \sum_{k=1}^\infty \epsilon 2^{-k}.
	\end{equation*}
	Thus $f \in \mathcal{H}_{p'}^{t',r'} (^\ast X ) $.  Thus the proof is complete.
\end{proof}

Before proving the predual of $M_{p}^{t,r} (X)$,  we first recall the following result.
\begin{lemma}[Corollary 1.3.22, \cite{HVVW16}] \label{duality Banach value}
	Let $X$ be a Banach space. Let $(S, \mathcal A, \mu )$ be a $\sigma$-finite  measure space and let $X$ be reflexive or $X^\ast$ be separable. Then for all $1\le p <\infty $ we have an isometric 	isomorphism
	\begin{equation*}
		(L^p (S, X) ) ^* = L^{p'} (S,  X^\ast), \quad  1/p +1/p' =1.
	\end{equation*}
\end{lemma}
Now we are ready to show that
the predual space of  $M_{p}^{t,r} (X)$ is $\mathcal{H}_{p'}^{t',r'} (^\ast X)$.
\begin{theorem} \label{predual seq BM}
	Let $X$ be a Banach space  such that $  ( ^\ast X )^ \ast = X $.
	Let $^\ast X$ be reflexive or $X$ be separable. Let $1 < p < t < r <\infty$  or $1 <  p \le t < r =\infty $. Then the dual  space of
	$\mathcal{H}_{p'}^{t',r'} (^\ast X)$, denoted by $ \Big( \mathcal{H}_{p'}^{t',r'} (^\ast X) \Big)^* $, is
	$M_{p}^{t,r} (X)$, that is,  
	\begin{equation*}
		\Big( \mathcal{H}_{p'}^{t',r'} (^\ast X) \Big)^*   = M_{p}^{t,r} (X)
	\end{equation*}
	in the following sense:
	
	{\rm (i)} if $f \in M_{p}^{t,r} (X)$, then the linear functional 
	\begin{equation} \label{Jf}
		J_f : f \to  J_f (g) :=   \int_\rn  \langle g (x), f (x) \rangle \d x
	\end{equation}
	is bounded on $ \mathcal{H}_{p'}^{t',r'} (^\ast X) $.
	
	{\rm (ii)} conversely, any continuous linear functional on $\mathcal{H}_{p'}^{t',r'} (^\ast X)$ arises as in  (\ref{Jf}) with a unique $f \in M_{p}^{t,r} (X)$.
	
	Moreover, $\|f\|_{  M_{p}^{t,r} (X)}  = \| J_f \|_{ ( \mathcal{H}_{p'}^{t',r'} (^\ast X) )^*}$ and 
	\begin{equation} \label{H = max M le 1}
		\|g\|_{\mathcal{H}_{p'}^{t',r'} (^\ast X)  } =\max \left\{\left|\int_\rn  \langle g (x), f (x) \rangle \d x    \right|: f\in  M_{p}^{t,r} (X), \|f\|_{  M_{p}^{t,r} (X) } \le 1 \right\}.
	\end{equation}
\end{theorem}
\begin{proof}
	(i) Let $g \in \mathcal{H}_{p'}^{t',r'} (^\ast X)$ and $\epsilon>0$. Then $g = \sum_{(j,k)\in\mathbb{Z}^{n+1}}\lambda_{j,k}b_{j,k}$ with $\lambda=\{\lambda_{j,k}\}_{(j,k)\in\mathbb{Z}^{n+1}}\in\ell^{r'}(\mathbb{Z}^{n+1})$
	and $g_{jk}$ is a  $(p',t', {^\ast X})$-block supported on $Q_{j,k}$ such that 
	\begin{equation*}
		\left( \sum_{(j,k)\in\mathbb{Z}^{n+1}} |\lambda_{j,k}|^{r'}\right)^{1/r'}  \le (1+\epsilon) \| g\|_{ \mathcal{H}_{p'}^{t',r'} (^\ast X)}.
	\end{equation*}
	By  H\"older's inequality,
	\begin{align} \label{f g le M H}
		\nonumber
		\left| \int_\rn \langle g (x), f (x) \rangle \d x\right| &  = 	\left|	\int_\rn  \left\langle \sum_{(j,k)\in\mathbb{Z}^{n+1}}\lambda_{j,k}b_{j,k} (x) ,  f (x) \right\rangle \d x\right| \\
		\nonumber
		& =\left| \sum_{(j,k)\in\mathbb{Z}^{n+1}}\lambda_{j,k} \int_{Q _{j,k} } \langle b_{j,k} (x) , f (x) \rangle \d x \right| \\
		\nonumber
		& \le \sum_{(j,k)\in\mathbb{Z}^{n+1}} |\lambda_{j,k}| \int_{Q _{j,k} }   \| b_{j,k} (x)\|_{^\ast X} \| f (x)\|_{X} \d x \\
		\nonumber
		& \le \sum_{(j,k)\in\mathbb{Z}^{n+1}} |\lambda_{j,k}| |Q|^{1/t-1/p} \left( \int_{Q _{j,k} } \| f (x)\|_{X}^p   \d x \right)^{1/p}\\
		\nonumber
		& \le \left( \sum_{(j,k)\in\mathbb{Z}^{n+1}} |\lambda_{j,k}|^{r'}\right)^{1/r'} \| f\|_{M_{p}^{t,r} (X) } \\
		& \le (1+\epsilon) \| g\|_{ \mathcal{H}_{p'}^{t',r'} (^\ast X)}  \| f\|_{M_{p}^{t,r} (X) }.
	\end{align}
	Letting $\epsilon\to 0^+$,  we prove (i).
	
	(ii) Let $L$ be a continuous linear functional on $\mathcal{H}_{p'}^{t',r'} (^\ast X) $. For any $j \in \mathbb Z$, $k \in \mathbb Z^n$ and for any  $^\ast X$-valued function $ g  \in L^{p'} (^\ast X) (Q_{j,k}) $ (the norm of $L^{p'} (^\ast X) (Q_{j,k})$ is $\|g\|_{L^{p'} (^\ast X) (Q_{j,k})}  = ( \int_{Q_{j,k} } \| g(x)\|_{ ^\ast X} ^{p'}\d x )^{1/p'}$), the mapping
	\begin{equation} \label{L jk}
		L_{j,k} (g) : = L ( g \chi_{Q_{j,k}} )
	\end{equation}
	is a continuous linear functional on $ L^{p'} (^\ast X) (Q_{j,k})$. Then using the condition that $^\ast X$ is reflexive or $X$ is separable and Lemma \ref{duality Banach value}, there exists  a unique $X$-valued function $f_{j,k} \in L^p (X) (Q_{j,k})$ such that for any $g \in L^{p'} (^\ast X) (Q_{j,k})$,
	\begin{equation*}
		L_{j,k} (g) = \int_{Q_{j,k} } \langle g (x),  f_{j,k} (x) \rangle \d x .
	\end{equation*}
	Thus there exists a $X$-valued function $f \in L^p_{\operatorname{loc}} (X) $ such that $f_{j,k} = f \chi_{Q_{j,k}} $ for each $j \in \mathbb Z, k\in \mathbb Z^n$.  Then for any given $j \in \mathbb Z, k\in \mathbb Z^n$ and for $g \in  L^{p'} (^\ast X) (Q_{j,k})$,
	\begin{equation} \label{L = f}
		L_{j,k} (g)  = \int_{Q_{j,k} } \langle g (x), f (x) \rangle \d x.
	\end{equation}
	Next we show that $f \in M_{p}^{t,r} (X) $. 
	Fix $\epsilon>0$ and a finite set $K \subset \mathbb Z^{n+1}$. For each $(j,k) \in K$, there exists $  g_{j,k,\epsilon}  \in L^{p'} (^\ast X) (Q_{j,k})$ such that $ \left\|  \chi_{Q_{j,k}} \|g_{j,k,\epsilon}\|_{^\ast X }  \right\|_{L^{p'}  }   \le 1 $ and 
	\begin{equation} \label{f le f g jk epsilon}
		\left\|  \chi_{Q_{j,k}}  \|f \|_{X }  \right\|_{L^p } \le (1+\epsilon) \int_{Q_{j,k} } \langle g_{j,k,\epsilon} (x), f (x)  \rangle  \d x .
	\end{equation}
	Note that $ |Q_{j,k}| ^{1/t - 1/p}  g_{j,k,\epsilon} $  is a $(p' , t', {^\ast X}) $-block supported on $Q_{j,k}$  since $ \left\|  \chi_{Q_{j,k}} \|g_{j,k,\epsilon}\|_{^\ast X }  \right\|_{L^{p'}  }   \le 1 $. Take an arbitrary nonnegative
	sequence $\{  \lambda_{j,k}\} \in \ell^{r'} (\mathbb Z^{1+n})$
	supported on $K$ and set
	\begin{equation}\label{g K decom}
		g_{K,\epsilon}  = \sum_{(j,k) \in K}  \lambda_{j,k}  |Q_{j,k}| ^{1/t - 1/p}  g_{j,k,\epsilon}  \in \mathcal{H}_{p'}^{t',r'} (^\ast X) .
	\end{equation}
	Then  from (\ref{L = f}), (\ref{L jk}), (\ref{f le f g jk epsilon}), the fact that $K$ is a finite set, and the linearity of $L$, we get 
	\begin{align*}
		&	\sum_{(j,k) \in K}  \lambda_{j,k}  |Q_{j,k}| ^{1/t - 1/p} \left\|  \chi_{Q_{j,k}}  \|f \|_{X }  \right\|_{L^p } \\
		& \le  (1+\epsilon) \sum_{(j,k) \in K}  \lambda_{j,k}  |Q_{j,k}| ^{1/t - 1/p} \int_{Q_{j,k} } \langle g_{j,k,\epsilon} (x), f (x)  \rangle  \d x \\
		& =  (1+\epsilon) \int_\rn  \langle g_{K,\epsilon} , f(x ) \rangle  \d x \\
		& = (1+\epsilon) L (g_{K,\epsilon}  ).
	\end{align*}
	By the decomposition (\ref{g K decom}) and $L$ is a continuous linear functional on $ \mathcal{H}_{p'}^{t',r'} (^\ast X)$, we obtain
	\begin{equation*}
		L (g_{K,\epsilon}  )  \le  \| L \|_{ \mathcal{H}_{p'}^{t',r'} (^\ast X) \to \mathbb R }  \| g_{K,\epsilon} \|_{  \mathcal{H}_{p'}^{t',r'} (^\ast X) } \le \| L \|_{ \mathcal{H}_{p'}^{t',r'} (^\ast X) \to \mathbb R } \left( \sum_{(j,k) \in K}   \lambda_{j,k} ^ {r'} \right)^{1/r'}.
	\end{equation*}
	Since $r >1$ and $K  \subset \mathbb Z^{1+n}$ and $ \{\lambda_{j,k} \}_{(j,k) \in K } $ are arbitrary, we conclude that
	\begin{align*}
		& \| f \|_{M_{p}^{t,r} (X) }  \\
		& = \sup \left\{  \sum_{ (j,k) \in \mathbb Z^{1+n}  }  \lambda_{j,k} |Q_{j,k} |^{1/t-1/p}  	\left\|  \chi_{Q_{j,k}}  \|f \|_{X}  \right\|_{L^p } :  \| \{\lambda_{j,k} \} \|_{\ell^{r' } (\mathbb Z^{1+n})}  = 1     \right\} \\
		&  \le  (1+\epsilon)\| L \|_{ \Big(  \mathcal{H}_{p'}^{t',r'} (^\ast X) \Big)^*  }  < \infty .
	\end{align*}
	Thus $f \in M_{p}^{t,r} (X)$. Let $\epsilon \to 0^{+}$, we have $\| f \|_{M_{p}^{t,r} (X) } \le \| L \|_{ \Big(  \mathcal{H}_{p'}^{t',r'} (^\ast X) \Big)^*  } $. Together (i), we obtain $\| f \|_{M_{p}^{t,r} (X) } = \| L \|_{ \Big(  \mathcal{H}_{p'}^{t',r'} (^\ast X) \Big)^*  } $. 
	
	Next we prove that $L$ arises as in (\ref{Jf}) with $f \in  M_{p}^{t,r} (X)$. For any $b \in  \mathcal{H}_{p'}^{t',r'} (^\ast X)$, there exist sequences $ \{ \lambda_{j,k}\}_{(j,k) \in \mathbb Z^{1+n}} \subset \mathbb R$  and  $\{ b_{j,k}  \}_{(j,k) \in \mathbb Z^{1+n} }$ where  $b_{j,k}$ a  $(p',t', {^\ast X})$-block supported on $Q_{j,k}$ such that 
	\begin{equation}
		b = \sum_{ (j,k) \in \mathbb Z^{1+n} } \lambda_{j,k} b_{j,k}
	\end{equation}
	almost everywhere on $\rn$  and
	\begin{equation*}
		\left(  \sum_{ (j,k) \in \mathbb Z^{1+n} }  | \lambda_{j,k} |^{r'}\right)^{1/r'}  \le (1+\epsilon) \|b\|_{ \mathcal{H}_{p'}^{t',r'} (^\ast X) } .
	\end{equation*}
	For any given $N \in \mathbb N$, let 
	\begin{equation*}
		b_N : = \sum_{ |j| \le N }  \sum_{ |k| \le N } \lambda_{j,k} b_{j,k};
	\end{equation*}
	here and in what follows, for any $k : = (k_1, \ldots, k_n) \in \mathbb Z^n  $, $ |k| = |k_1| + \ldots + |k_n|$. Since $r ' \in [1,\infty) $, for any $\eta >0$, there exists $K \in \mathbb N$ such that, for $N >K$, we have
	\begin{equation*}
		\| b-b_N\|_{\mathcal{H}_{p'}^{t',r'} (^\ast X)} \le  \left(  \sum_{ |j| > N }  \sum_{ |k| \le N }  | \lambda_{j,k} |^{r'}\right)^{1/r'}  +  
		\left(   \sum_{ |j| \le N }  \sum_{ |k| > N  }  \lambda_{j,k} |^{r'} \right)^{1/r'}  <\eta.
	\end{equation*}
	Thus, $b_N \to b$ in $\mathcal{H}_{p'}^{t',r'} (^\ast X)$ as $N \to \infty$. From this, both the continuity and the linearity of $L$,  (\ref{L = f}),  we get, for each $b \in \mathcal{H}_{p'}^{t',r'} (^\ast X)$,
	\begin{align}\label{L(b) = lim b_N}
		\nonumber
		L(b) & = \lim_{N \to \infty} L(b_N) = \lim_{N \to \infty}  \sum_{ |j| \le N }  \sum_{ |k| \le N } \lambda_{j,k} L ( b_{j,k} ) \\
		& =  \lim_{N \to \infty}  \sum_{ |j| \le N }  \sum_{ |k| \le N } \lambda_{j,k} \int_\rn \langle b_{j,k} (x), f (x)  \rangle \d x =  \lim_{N \to \infty} \int_\rn \langle b_N (x)  , f (x) \rangle \d x.
	\end{align} 
	Similarly to the estimation of (\ref{f g le M H}), using the fact that $b_N \to b$ in $\mathcal{H}_{p'}^{t',r'} (^\ast X)$ as $N \to \infty$, we obtain
	\begin{align*}
	&	\lim_{N \to \infty} \left|   \int_\rn \langle b (x) ,  f(x) \rangle \d x - \int_\rn \langle b_N (x) , f(x) \rangle \d x  \right| \\
	&  \le \lim_{N \to \infty}  \|b-b_N\|_{ \mathcal{H}_{p'}^{t',r'} (^\ast X)}  \| f\|_{M_{p}^{t,r} (X) }  =0.
	\end{align*}
	From this and (\ref{L(b) = lim b_N}), we obtain $L$ arises as in (\ref{Jf}) with $f \in M_{p}^{t,r} (X)$.
	
	Next we show that $f$ is unique. Suppose that there exists another $\tilde f \in M_{p}^{t,r} (X)$ such that $L$ arises as in (\ref{Jf}) with $f $ replaced by $\tilde f$.
	Since
	\begin{equation*}
		\| \tilde f (x) - f(x)\|_{X} = \sup_{  \| h\|_{^\ast X} =1 } \langle  h ,   \tilde f (x)- f(x)\rangle, 
	\end{equation*}
	for fixed $\epsilon>0$, there exists a $^\ast X$-valued function $g(x)$ with $\| g(x) \|_{ ^\ast X } =1$ such that 
	\begin{equation*}
		\|\tilde f (x) - f(x)\|_{X}  \le (1+\epsilon) \langle  g (x), \tilde f(x ) - f(x)   \rangle.
	\end{equation*}
	Note that 
	\begin{equation*}
		\| \chi_Q \|g\|_{^\ast X} \|_{L^{p'}} = \left( \int_Q  \|g (x) \|_{^\ast X}^p \d x \right)^{1/p} = |Q|^{1/p}.
	\end{equation*}
	So $\chi_Q g $ is a $ (p', t', {^\ast X}) $-block supported on $Q$ modulo a multiplicative constant. Thus
	\begin{align*}
		0 = \int_Q \langle  g (x), \tilde f(x ) - f(x) \rangle \d x \ge \frac{1}{1+\epsilon} \int_Q \|\tilde f (x) - f(x)\|_{X} \d x .
	\end{align*}
	This, combined with the arbitrariness of $Q\in \D$, further implies that $ \tilde{f} = f $ almost everywhere.

	Using  \cite[Theorem 87, Existence of the norm attainer]{SFH20}, (\ref{H = max M le 1}) is included in what we have proven.
	Thus we complete the proof.
\end{proof}

The block space has  translation invariance. This property can be used to show that the convolution operator with $L^1$-functions is bounded in block space. 
\begin{lemma} \label{trans invar H}
	Let $X$ be a Banach space  such that $  ( ^\ast X )^ \ast = X $. Let $^\ast X$ be reflexive or $X$ be separable.
	Let $1< p<t<r<\infty$
	or $1< p\le t<r=\infty$.
	For any $y \in \rn$, $f\in \H_{p'}^{t',r'}(^\ast X)$, we have
	\begin{equation*}
		\| \tau_y f\|_{ \H_{p'}^{t',r'}(^\ast X)  }  \lesssim 	\| f \|_{ \H_{p'}^{t',r'}(^\ast X)  }.
	\end{equation*}
	Here and below,  $\tau_y f   =  f (\cdot -y)  $.
\end{lemma}
\begin{proof}
	Let $f \in\H_{p'}^{t',r'}(^\ast X) $. Then there exist a sequence $\{ \lambda_{j,k} \}_{ (j,k) \in \mathbb Z ^{1+n} }  \in \ell ^{r'} $  and a sequence  $(p',t',{^\ast X})$-block $\{  b_{j,k} \}  _{ (j,k) \in \mathbb Z ^{1+n} } $  supported on  $Q_{j,k}   $
	such that $f 	=\sum_{(j,k)\in\mathbb{Z}^{n+1}}\lambda_{j,k}b_{j,k} $  and $\| \lambda_{j,k} \|_{\ell ^{r'}}  \le \| f\|_{\H_{p'}^{t',r'}(^\ast X)  }    +\epsilon$.
	Then 
	\begin{equation*}
		\tau_y f (x) = \sum_{(j,k)\in\mathbb{Z}^{n+1}}\lambda_{j,k}  \tau_y b_{j,k} (x).
	\end{equation*}
	$\tau_y b_{j,k}  = b_{j,k} (\cdot -y) $ is supported on $ Q_{j,k} +y := \{ z +y : z \in Q_{j,k}  \} $. There are at most $2^n$ cubes $Q_{i, j, k} \in \D _ j$ (we use there indexes $i, j, k$ to mean that $Q_{i, j, k}$ is related to cube $Q_{j,k}$) such that
	\begin{equation*}
		Q_{j,k} +y \subset  \bigcup_{i=1}^{2^n} Q_{i,j,k}.
	\end{equation*}
	Hence 
	\begin{equation*}
		\tau_y f (x) = \sum_{(j,k)\in\mathbb{Z}^{n+1}} \sum_{i=1}^{2^n} \lambda_{j,k} b_{j,k} (x -y ) \chi_{ Q_{i, j, k} } (x).
	\end{equation*}
	The function $b_{j,k} (\cdot -y ) \chi_{ Q_{i, j, k} }$ is supported on $Q_{i, j, k} \in \D _ j $ and $ \|b_{j,k} (\cdot -y ) \chi_{ Q_{i, j, k} } \|_{L^{p'}  (^\ast X) }  \le \| b_{j,k}\|_{ L^{p'}  (^\ast X) }  \le |Q_{j,k}|^{  1/t -1/p}  =   | Q_{i, j, k}|^{  1/t -1/p}   $. Hence  $b_{j,k} (\cdot -y ) \chi_{ Q_{i, j, k} }$ is a $(p',t', {^\ast X})$-block.
	Hence we obtain
	\begin{equation*}
		\| 	\tau_y f \| _{ \mathcal{H}_{p'}^{t',r'} (^\ast X) }    \le  2^{n/r'} \| \lambda_{j,k} \|_{\ell ^{r'}}  \le 2^{n/r'} \left( \| f\|_{\H_{p'}^{t',r'}(\ell^{q'})  }    +\epsilon \right) .
	\end{equation*}
	Letting $\epsilon \downarrow 0$, we obtain the result. 
\end{proof}

\begin{remark} \label{tri-angle}
	Using Theorem \ref{predual seq BM}, we have
	\begin{align*}
	&	\| f +g \|_{ \H_{p'}^{t',r'}(^\ast X) }\\
	 & = \max_{ \|h\|_ { M _p^{t,r} (X) }  \le 1  } \left| \int_\rn \langle (f (x)+g (x) ), h (x) \rangle \d x \right| \\
		& \le  \max_{ \|h\|_ { M _p^{t,r} (X)   }  \le 1  } \left| \int_\rn \langle f (x) , h (x)\rangle \d x \right| +  \max_{ \|w\|_ { M _p^{t,r}  (X)  }  \le 1  } \left| \int_\rn \langle g (x), w (x) \rangle \d x \right|  \\
		& = 	\| f  \|_{ \H_{p'}^{t',r'}(^\ast X) } +\| g  \|_{ \H_{p'}^{t',r'}(^\ast X) }.
	\end{align*}
\end{remark}
For scalar function $g \in L^1 (\rn)$ and $^\ast X$-valued function $f$,  we define the convolution of $f$ and $g$ by
\begin{equation*}
	f * g (x) = \int_\rn  g(x-y)f(y) \d y.
\end{equation*} 
Using Lemma \ref{trans invar H} and the triangle inequality for $ \H_{p'}^{t',r'}(^\ast X) $ (Remark \ref{tri-angle}), we obtain the following convolution inequality.

\begin{corollary}
	Let $X$ be a Banach space  such that $  ( ^\ast X )^ \ast = X $. Let $^\ast X$ be reflexive or $X$ be separable.
	Let $1< p<t<r<\infty$
	or $1< p\le t<r=\infty$. Then for all scalar functions $g \in L^1 (\rn)$ and $f \in \H_{p'}^{t',r'}(^\ast X)$, we have
	\begin{equation*}
		\|  g *f\|_{  \H_{p'}^{t',r'}(^\ast X)  }  \lesssim \|g\|_{L^1}  \| f\|_{  \H_{p'}^{t',r'}(^\ast X)  },
	\end{equation*}
	where the  implicit constant is the same as Lemma \ref{trans invar H}. 
\end{corollary}

\begin{proof}
	Let  $g \in L^1 (\rn)$ and $f \in \H_{p'}^{t',r'}(^\ast X)$. Then by Theorem \ref{predual seq BM} and Lemma \ref{trans invar H},  we obtain
	\begin{align*}
		\|  g *f\|_{  \H_{p'}^{t',r'}(^\ast X)  } & = \left\| \int_{\rn} g (y)  f (\cdot-y) \d y        \right\|_{  \H_{p'}^{t',r'}(^\ast X)  } \\
		&  = \max_{ \|h\|_{ M_p^{t,r} (X)  }  \le  1  }   \left| \int_\rn  \left\langle  \int_{\rn} g (y)  f (x-y) \d y , h(x)  \right\rangle  \d x \right| \\
		&= \max_{ \|h\|_{ M_p^{t,r}  (X) }  \le  1  }   \left| \int_\rn g (y)  \int_{\rn}  \langle  f (x-y), h(x) \rangle  \d x  \d y    \right| \\
		& \le   \left| \int_\rn g (y)  \max_{ \|h\|_{ M_p^{t,r} (X)  }  \le  1  } \int_{\rn} \langle  f (x-y),   h(x)  \rangle \d x \d y    \right| \\
		& \le  \int_{\rn}  |g(y)| \left\|    f (\cdot-y)        \right\|_{  \H_{p'}^{t',r'}(^\ast X)  }  \d y \\
		& \lesssim \left\|    f         \right\|_{  \H_{p'}^{t',r'}(^\ast X)  } \int_{\rn}  |g(y)|   \d y. \qedhere
	\end{align*}
\end{proof}

\section{Properties of  the block space}\label{property block}

In this section, we  first discuss completeness,  denseness and Fatou property of 
the function space $\mathcal{H}_{p'}^{t',r'}(^\ast X) $.  Since we only suppose that  $^\ast X$ is reflexive  (or $X$ is a separable Banach space), it is difficult to show whether  $^\ast X$-valued block spaces $\mathcal{H}_{p'}^{t',r'}(^\ast X) $ have  the lattice property.
But, if $X = \ell^q, q\in (1,\infty)$, we obtain block spaces $\mathcal{H}_{p'}^{t',r'}(\ell^{q'}) $ have  the lattice property.

\subsection{The completeness of the block space} \label{banach space}

We first show the completeness.

\begin{theorem} \label{Banach}
	Let $X$ be a Banach space  such that $  ( ^\ast X )^ \ast = X $. Let $^\ast X$ be reflexive or $X$ be separable.
	Let $1< p<t<r<\infty$ 	or $1<p\le t<r=\infty$. Then $\H_{p'}^{t',r'}(^\ast X) $ is a Banach space.
\end{theorem}

\begin{proof}
	We only  prove the completeness  since others are easy. 
	
	Let $ \{ f_i\}_{i\in \mathbb N}$  be a sequence such that for each $i \in \mathbb N$, $f_i \in \H_{p'}^{t',r'}(^\ast X)  $, $ \sum_{i\ge 1} \| f_i \|_{ \H_{p'}^{t',r'}(^\ast X)  }  <\infty  $.
	From  \cite[Theorem 2.17]{M16} (scalar version of Theorem \ref{predual seq BM}),
	for any ball $B$, we have
	\begin{equation*}
		\int_{B} \sum_{i \ge 1} \| f_i (y) \|_{^\ast X } \d y =\sum_{i \ge 1} \int_{B}  \| f_i (y) \|_{^\ast X } \d y \le \| \chi_B \|_{  M _p^{t,r}   }  \sum_{i \ge 1} \|    f_i \|_{ 	\H_{p'}^{t',r'}(^\ast X )  }.
	\end{equation*}
	Form $\|  \sum_{i \ge 1}  f_i \|_{ ^\ast X }   \le \sum_{i \ge 1} \| f_i  \|_{^\ast X  }$,
	we obtain  $ f = \sum_{i \ge 1}  f_i  $ is a well defined, $^\ast X $-valued Lebesgue measurable function and  $f\in  L^1_{\operatorname{loc}} ( ^\ast X ) $. So we only need to show that  $ f = \sum_{i \ge 1}  f_i  $ belongs to $\H_{p'}^{t',r'}(^\ast X ) $.
	Fix $\epsilon > 0$. There exists a positive integer $N_0$ such that for all $N \ge N_0$, there holds
	\begin{equation*}
		\sum_{i \ge N } \|    f_i \|_{ 	\H_{p'}^{t',r'}( ^\ast X )  }  <\epsilon.
	\end{equation*}
	For this $\epsilon > 0$, there exists a decomposition 
	\begin{equation*}
		f _i=  \sum_{(j,k)\in \mathbb Z ^{n+1}} \lambda_{i,j,k} b_{i,j,k},
	\end{equation*}
	where  $b_{i,j,k}$ is a  $(p',t', ^\ast X )$-block supported on $Q_{j,k}$ and 
	\begin{equation*}
		\left(    \sum_{ (j,k)\in \mathbb Z ^{n+1}  }   | \lambda_{i,j,k}| ^{r'} \right)^{1/r'}  \le (1+\epsilon)  \|    f_i \|_{ 	\H_{p'}^{t',r'}(^\ast X )  }.
	\end{equation*}
	Furthermore, for any $ 1\le i \le N_0 $, there exists a index set $ M _i  \subset  \mathbb Z ^{n+1} $  such that 
	\begin{equation*}
		\bigg\|  f_i  - \sum_{(j,k)\in M_i} \lambda_{i,j,k} b_{i,j,k}  \bigg\|_{ 	\H_{p'}^{t',r'}(^\ast X )  } \le \bigg(  \sum_{(j,k)\in   \mathbb Z ^{n+1}  \backslash  M_i} |\lambda_{i,j,k}|^{r'} \bigg) ^{1/r'} < 2^{-i} \epsilon.
	\end{equation*}
	Therefore, for all ball $B$, 
	\begin{align*}
		&	\int_B \bigg\| f(y)  -  \sum_{i=1}^{N_0}  \sum_{(j,k)\in M_i} \lambda_{i,j,k} b_{i,j,k} (y) \bigg\|_{^\ast X }  \d y \\
		& \le \int_B  \bigg\| f(y) -\sum_{i=1}^{N_0} f_i (y)  \bigg\|_{^\ast X } \d y + \int_B  \bigg\| \sum_{i=1}^{N_0} f_i (y) - \sum_{i=1}^{N_0}  \sum_{(j,k)\in M_i} \lambda_{i,j,k} b_{i,j,k} (y)  \bigg\|_{^\ast X } \d y  \\
		& \le  \int_B  \bigg\| \sum_{i=N_0+1}^{\infty } f_i (y)  \bigg\|_{^\ast X } \d y +  \sum_{i=1}^{N_0} \int_B  \bigg\|  f_i (y) -   \sum_{(j,k)\in M_i} \lambda_{i,j,k} b_{i,j,k} (y)  \bigg\|_{^\ast X } \d y  \\
		& \le \| \chi_B \|_{M_p^{t,r} } \bigg( \sum_{i=N_0+1}^{\infty }  \|f_i\|_{ \H_{p'}^{t',r'}(^\ast X )   }  +   \sum_{i=1}^{N_0}  \bigg\|  f_i  -   \sum_{(j,k)\in M_i} \lambda_{i,j,k} b_{i,j,k}  \bigg\|_{\H_{p'}^{t',r'}(^\ast X )  }  \bigg) \\
		& \le \| \chi_B \|_{M_p^{t,r} } ( \epsilon +\sum_{i=1}^{N_0}  2^{-i} \epsilon  ) \\
		& \lesssim  \| \chi_B \|_{M_p^{t,r} }  \epsilon.
	\end{align*}
	As a consequence,  $\sum_{i\ge 1} \sum_{(j,k)\in \mathbb Z ^{n+1}} \lambda_{i,j,k} b_{i,j,k}$ converges to $f$ in  $L^1_{\operatorname{loc}} (^\ast X ) $.  Hence,  $\sum_{i\ge 1} \sum_{(j,k)\in \mathbb Z ^{n+1}} \lambda_{i,j,k} b_{i,j,k}$ converges to $f$ locally in measure. Therefore, there exists a subsequence of  
	\begin{equation*}
		\left\{\sum_{i= 1}^{N} \sum_{(j,k)\in K \subset \mathbb Z ^{n+1} , \sharp K = M} \lambda_{i,j,k} b_{i,j,k}  \right\}_{ N \in \mathbb N ,M \in \mathbb N}
	\end{equation*}
	converges to $f$ a.e. 
	Recall that $\sharp K$ is the cardinal number of the set $K$.
	Furthermore, $\lambda_{i,j,k}$, $i\in \mathbb N, (j,k) \in \mathbb Z ^{n+1}$  satisfies
	\begin{equation*}
		\sum_{i\ge 1}	\bigg(    \sum_{ (j,k)\in \mathbb Z ^{n+1}  }   | \lambda_{i,j,k}| ^{r'} \bigg)^{1/r'}  \le \sum_{i\ge 1} (1+\epsilon)  \|    f_i \|_{ 	\H_{p'}^{t',r'}(^\ast X )  }.
	\end{equation*}
	That is, $ \sum_{i\ge  1}  f_i$ converges to $f$ in  $	\H_{p'}^{t',r'}(^\ast X)$. Since $\epsilon$ is arbitrary, we obtain
	\begin{equation*}
		\| f\|_{ \H_{p'}^{t',r'}(^\ast X)   }  =  	\bigg\| \sum_{i\ge  1}  f_i \bigg\|_{ \H_{p'}^{t',r'}(^\ast X)   }  \le \sum_{i\ge  1} \|    f_i \|_{ 	\H_{p'}^{t',r'}(^\ast X)  }.
	\end{equation*}
	Then by the well known result that a normed linear space is compete if and only if every absolutely summable sequence is summable (for example, see \cite[Theorem III.3]{RS72}), 	$\H_{p'}^{t',r'}(^\ast X) $ is complete.
\end{proof}

\subsection{Density of $L^{p'}  _c (^\ast X)$  and $C_c^\infty (^\ast X)$}
Let $L^{p'}  _c (^\ast X)$ be  the set of  compactly supported,  $p'$-integrable, $^\ast X$-valued  functions on $\rn$.
Let $C_c^\infty (^\ast X)$ be the smooth, compactly supported, $^\ast X$-valued  functions on $\rn$. 
Next we obtain  $L^{p'}  _c (^\ast X)$  and $C_c^\infty (^\ast X)$ are  dense subspaces of $\H_{p'}^{t',r'}(^\ast X)  $.

\begin{lemma}
	[Proposition 1.2.32, \cite{HVVW16}]  \label{Banach Mollifiers}  Let $X$ be a Banach space.
	Let $f \in L^p (\rn; X )$ for some $p\in [1,\infty)$ and $\phi\in L^1 (\rn)$. For $\epsilon >0$, define $\phi_\epsilon (y) = \epsilon^{-n} \phi (\epsilon^{-1} y) $. Then 
	\begin{equation*}
		\phi_\epsilon * f \to c_\phi  f 
	\end{equation*} 
	in $L^p (\rn; X )$  as $\epsilon \downarrow 0$, where $c_\phi : = \int_\rn \phi (y) \d y.$
\end{lemma}
\begin{theorem}
	Let $X$ be a Banach space  such that $  ( ^\ast X )^ \ast = X $. Let $^\ast X$ be reflexive or $X$ be separable.
	Let $1 < p<t<r<\infty$ 	or $1< p\le t<r=\infty$. Then $L^{p'}  _c (^\ast X)$ is dense in $\H_{p'}^{t',r'}(^\ast X)  $. In particular, by mollification, $C_c^\infty (^\ast X)$ is dense in $\H_{p'}^{t',r'}(^\ast X)  $.
\end{theorem}
\begin{proof}
	We use the idea from \cite[Theorem 345]{SFH20}.
	Since $ f \in \H_{p'}^{t',r'}(^\ast X) $, there exist a sequence $\{ \lambda_{j,k} \}_{ (j,k) \in \mathbb Z ^{1+n} }  \in \ell ^{r'} $  and a sequence  $(p',t',^\ast X)$-block $\{  b_{j,k} \}  _{ (j,k) \in \mathbb Z ^{1+n} } $  supported on  $Q_{j,k}   $
	such that $f 	=\sum_{(j,k)\in\mathbb{Z}^{n+1}}\lambda_{j,k}b_{j,k} $  and $\| \{ \lambda_{j,k} \} \|_{\ell ^{r'}}  \le \| f\|_{\H_{p'}^{t',r'}(^\ast X)  }    +\epsilon$.	
	Define 
	\begin{equation} \label{f_N dense}
		f_N = \sum_{|(j,k)|_\infty  \le N}\lambda_{j,k}b_{j,k},
	\end{equation}
	where $|(j,k)|_\infty  = \max  \{ |j|, |k_1|, \ldots, |k_n| \}  $.
	Since $r' \in [1,\infty)$, we have
	\begin{equation*}
		\| f -f_N \|_{ \H_{p'}^{t',r'}(^\ast X) }  \le  \left( \sum_{|j| \le N ,|k|_\infty  > N} | \lambda_{j,k} |^{r'}  + \sum_{ |j| > N ,|k|_\infty  \le  N}| \lambda_{j,k} |^{r'}  \right)^{1/r'}   \to 0 
	\end{equation*}
as $N\to \infty.$
	Now we show $f_N \in L^{p' } ( ^\ast X)$. We consider two cases. 
	
	Case $r=\infty$. Then
	\begin{align*}
		\|f_N\|_{L^{p' } (  ^\ast X) }  & \le  \sum_{|(j,k)|_\infty  \le N} |\lambda_{j,k} |  \| b_{j,k} \|_{L^{p' } ( ^\ast X)} \\
		& \le  \sum_{|(j,k)|_\infty  \le N}|\lambda_{j,k} | |Q_{jk}|^{1/t -1/p}  \\
		& \le 2^{-N (1/t-1/p)  }  ( \| f\|_{\H_{p'}^{t',r'}( ^\ast X)  }    +\epsilon)
		<\infty .
	\end{align*}
	Thus $f_N \in L^{p' } (  ^\ast X)$.
	
	Case $r<\infty$. By  Minkowski's inequality and H\"older's inequality, we have
	\begin{align*}
		\|f_N\|_{L^{p' } (  ^\ast X ) } & \le  \sum_{|(j,k)|_\infty  \le N} |\lambda_{j,k} |  \| b_{j,k} \|_{L^{p' } ( ^\ast X  )} \\
		& \le  \sum_{|(j,k)|_\infty  \le N} |\lambda_{j,k} | |Q_{jk}|^{1/t -1/p} \\
		& \le  \left( \sum_{|(j,k)|_\infty  \le N} |\lambda_{j,k} | ^{r'} \right)^{1/r'}   \left( \sum_{|(j,k)|_\infty  \le N} |Q_{jk}|^{r/t -r/p} \right)^{1/r} \\
		& \le  ( \| f\|_{\H_{p'}^{t',r'}(^\ast X)  }    +\epsilon) \left( \sum_{\ell = -N} ^N \sum_{k_1=-N} ^N \cdots \sum_{k_n=-N} ^N 2^{-\ell  nr (1/t-1/p)} \right)^{1/r} \\
		& <\infty.
	\end{align*}
	Thus  $f_N \in L^{p' } (  ^\ast X)$. Hence
	$L^{p'}  _c ( ^\ast X)$ is dense in $\H_{p'}^{t',r'}( ^\ast X)  $.

	Now we show $C_c^\infty (^\ast X)$ is dense in $\H_{p'}^{t',r'}(^\ast X)  $.  Let $f_N$  be  (\ref{f_N dense}). For a fixed $\epsilon>0$, we can choose $N_0 \in \mathbb N$ such that for all $N \ge N_0$,
	\begin{equation*}
		\| f -f_N \|_{ \H_{p'}^{t',r'}(^\ast X) }  < \epsilon.
	\end{equation*} 
	Let	
	\begin{equation*}
		\omega (x) = 	\begin{cases}
			c e ^{ 1/ (|x|^2 -1) }, & \operatorname{if}  |x| < 1; \\
			0, & \operatorname{else},
		\end{cases}
	\end{equation*} 
	where the constant $c  >0 $  is such that $ \int_{B(0,1)}  \omega (x) \d x =1 $. Set $\omega _\eta (x)  = \eta^{-n} \omega ( \eta^{-1} x )$ for $\eta>0$.
	Then $\omega_\eta * f_N \in C_c^\infty (^\ast X) $   and 
	\begin{equation*}
		\omega_\eta *	f_N = \sum_{|(j,k)|_\infty  \le N}\lambda_{j,k} \omega_\eta * b_{j,k}.
	\end{equation*}
	Indeed, we will show $\partial^{\alpha} (\omega_\eta *	f_N   )   = \partial^{\alpha} (\omega_\eta) *	f_N   $ for all multi-index $\alpha \in \mathbb N_0 ^n$. 
	Let $e_j := (0,  \ldots, 0,1, 0,\ldots, 0) \in \mathbb N_0^n$ such that the $j$-th item is $1$  and other items equal $0$.
	By iteration it suffices to consider the case $\alpha = e_j $. Since $\omega \in \mathscr S (\rn)$,
	\begin{align*}
		&\lim_{s \to 0}	\frac{1}{s} ( \omega_\eta *	f_N (x + s  e_j ) -  \omega_\eta *	f_N (x  ) )  \\
		& =\lim_{s \to 0}	 \frac{1}{s} \int_\rn ( \omega_\eta (x + s  e_j  -y  ) -   \omega_\eta (x- y)  )  f_N (y) \d y \\
		& = \int_\rn  \partial^\alpha  \omega_\eta (x   -y  )   f_N (y) \d y \\
		& = (\partial^\alpha \omega_\eta ) *f_N (y).
	\end{align*}
	Since $ \lim_{\eta \to 0} \| \omega_\eta *  b_{j,k} -  b_{j,k} \|_{L^{p'} (^\ast X)}  =0$ (Lemma \ref{Banach Mollifiers}),
	there exists $\eta_0 >0 $ such that for all $0<\eta <\eta_0$, all $ |(j,k)|_\infty \le N_0  $, supp   $\omega_\eta * b_{j,k}  \subset 2 Q _{j,k}$ and 
	\begin{equation*}
		\epsilon^{-1} (2 N_0 +1)^{ (n+1) / r'}  \lambda_{j,k}  \|  \omega_\eta * b_{j,k} -  b_{j,k} \|_{L^{p'} (^\ast X)} \le     | 2Q _{j,k}|^{1/t - 1/p} .
	\end{equation*}
	Hence $ \epsilon^{-1} (N_0 +1)^{ (n+1) / r'}  \lambda_{j,k} ( \omega_\eta * b_{j,k} -  b_{j,k}) $ is a $ (p',t',{^\ast X}) $-block supported on $ 2Q _{j,k} $.
	Then 
	\begin{align*}
		f_{N_0}  - \omega_\eta *	f_{N_0} &=   \sum_{|(j,k)|_\infty  \le N_0}    \lambda_{j,k} \left( b_{j,k}  -  \omega_\eta * b_{j,k}  \right) \\
		& = \sum_{|(j,k)|_\infty  \le {N_0} }  \epsilon (2 N_0 +1)^{-(n+1)/r'  }     
		\\
		& \quad \times           \left(  \epsilon^{-1} ( 2 N_0 +1)^{ (n+1) / r'}  \lambda_{j,k} ( \omega_\eta * b_{j,k} -  b_{j,k})         \right)  . 
	\end{align*}
	Hence 
	\begin{equation*}
		\| f_N  - \omega_\eta *	f_N \|_{ \H_{p'}^{t',r'}(^\ast X) } \le  \epsilon ( 2 N_0 +1)^{-(n+1)/r'  }   \left( \sum_{|(j,k)|_\infty  \le N_0}   1     \right) ^{1/r'}  = \epsilon .
	\end{equation*}
	Then by  Minkowski's inequality, we obtain
	\begin{equation*}
		\| f  - \omega_\eta *	f_N \|_{ \H_{p'}^{t',r'}(^\ast X) } 	\le  \| f -f_N \|_{ \H_{p'}^{t',r'}(^\ast X) } + 	\| f_N  - \omega_\eta *	f_N \|_{ \H_{p'}^{t',r'}(^\ast X) }  \le 2 \epsilon.
	\end{equation*} 
	Since  $\epsilon$  is arbitrary, we prove that  $C_c^\infty (^\ast X)$ is dense in $\H_{p'}^{t',r'}(^\ast X)  $.
\end{proof}

From the above theorem, when we investigate $  \H_{p'}^{t',r'}(^\ast X)$, the space $L^{p'}  _c ( ^\ast X) $ is a useful space. When considering the action of  linear operators defined  and continuous on    $  \H_{p'}^{t',r'}(^\ast X)$ and $L^{p'}   ( ^\ast X) $, it will be helpful to have a finite decomposition in $L^{p'}  _c ( ^\ast X) $.
The following result says that each function $f \in L^{p'}  _c ( ^\ast X)$ has a finite admissible expression. That is, the sum is finite.
\begin{theorem}\label{finite decom}
	Let $X$ be a Banach space  such that $  ( ^\ast X )^ \ast = X $. Let $^\ast X$ be reflexive or $X$ be separable.
	Let $1<  p<t<r<\infty$ 	or $1< p\le t<r=\infty$. Then each $ f \in L^{p'}  _c (^\ast X) $ admits the finite decomposition $ f  = \sum_{v=1}^M \lambda_v b_v $ where $ \lambda_1, \lambda_2, \ldots, \lambda_M \ge 0 $ and each $ b_v $ is a $(p',t',{^\ast X})$-block. Furthermore, 
	\begin{equation*}
		\| f\|_{ \H_{p'}^{t',r'}(^\ast X) } \approx \inf_\lambda \left( \sum_{v=1}^M \lambda_v ^{r'} \right)^{1/r'},
	\end{equation*}
	where $ \lambda = \{ \lambda_v\}_{v=1}^M $ runs over all finite admissible expressions:
	\begin{equation*}
		f= \sum_{v=1}^M \lambda_v b_v,
	\end{equation*}
	$ \lambda_1, \lambda_2, \ldots, \lambda_M \ge 0 $ and each $ b_v $ is a $(p',t',{^\ast X})$-block for each $v = 1,2,\ldots ,M$.
\end{theorem}
\begin{proof}
	We use the idea from \cite[Theorem 346]{SFH20}.
	
	Let $f \in L^{p'}  _c (^\ast X) \subset \H_{p'}^{t',r'}(^\ast X)   $. We may assume that $ \| f \|_ { \H_{p'}^{t',r'}(^\ast X) } \le 1$ and that supp $f \subset Q^{(0)}$ for a dyadic cube $ Q^{(0)}$, since we can decompose $f$ into $2^n$ sums, each of which is supported on a closure of a quadrant.
	Then $f$ can be decomposed as
	\begin{equation*}
		f = \sum_{(j,k)\in\mathbb{Z}^{n+1}} \lambda_{j,k} b_{j,k}
	\end{equation*}
	where $ \left( \sum_{(j,k)\in\mathbb{Z}^{n+1}} |\lambda_{j,k}|^{r'} \right)^{1/r'} \le 1 +\epsilon $. 
	If $\lambda_{j,k} <0 $, let $ \lambda_{j,k} b_{j,k} = -\lambda_{j,k}  (- b_{j,k}) $. Then $-\lambda_{j,k} >0 $ and $- b_{j,k}$ is also a  $(p',t',{^\ast X})$-block. Hence, we may assume that  $\lambda_{j,k} $ is a nonnegative real number.
	Furthermore, we may assume that supp $b_{j,k}  \subset Q ^{(0)}.$
	By the monotone convergence theorem and the fact that supp $f \subset  Q ^{(0)}$, we have
	\begin{equation*}
		f = \lim_{\eta \downarrow 0} \sum_{  j: \eta < 2^{j} < \eta^{-1} } \sum_{Q_{j,k} \subset Q^{(0)}} \lambda_{j,k} b_{j,k} \chi_{Q^{0}}
	\end{equation*} 
	in $L^{p'} (^\ast X)$. Thus if we let 
	\begin{equation*}
		g :=  \sum_{  j: 2^{j} \le \eta \; \operatorname{or} \;  2^{j} \ge \eta^{-1} } \sum_{Q_{j,k} \subset Q^{(0)}} \lambda_{j,k} b_{j,k} \chi_{Q^{0}},
	\end{equation*}
	then since $f\in L^{p' } (^\ast X)$, we obtain that $g$ is  a $(p',t',{^\ast X}) $-block as long as $\eta$ is small, so that
	\begin{equation*}
		f= g + \sum_{  j: \eta < 2^{j} < \eta^{-1} } \sum_{Q_{j,k} \subset Q^{(0)}} \lambda_{j,k} b_{j,k} \chi_{Q^{0}}
	\end{equation*} 
	is the the desired finite decomposition. Since the norm of $g$ can be made as small as we wish, we have the desired norm equivalence:
	\begin{equation*}
		\| f \|_{ \H_{p'}^{t',r'}(^\ast X) } \approx  \inf_{\lambda} \left( \sum_{v=1}^M \lambda_v ^{r'} \right)^{1/r'}.
	\end{equation*}
	Thus the proof is complete.
\end{proof}

\subsection{Fatou property of $\H_{p'}^{t',r'}( ^\ast X)$ }
In this subsection, the main result is the following theorem, which is called the Fatou property of $\H_{p'}^{t',r'}( ^\ast X)$. Before proving it, we first recall the following lemma.

Using Banach and Alaoglu Theorem (see, for example \cite[Theorem 89]{SFH20}) and Lemma \ref{duality Banach value}, we obtain the following result.
\begin{lemma}\label{weak* compct}
	Let $X$ be a Banach space. Let $(S, \mathcal A, \mu )$ be a $\sigma$-finite  measure space and let $X$ be reflexive or $ X ^\ast$ be separable. Let  $1< p <\infty $ and $\{ f_j \}_{j=1}^\infty$  be a bounded sequence in $L^p (S, X) $. Namely, there exists $M>0$ such that $ \|f_j\|_{L^p (S, X)} \le M$. Then we can find $f \in L^p (S, X)$  and a subsequence $\{ f_{j_k} \}_{k=1}^\infty$  such that 
	\begin{equation*}
		\lim_{k\to \infty} \int_S \langle f_{j_k}(x) , g (x) \rangle \d \mu (x) = \int_S \langle f(x) , g (x) \rangle\d \mu (x)
	\end{equation*}
	for all $g \in L^{p'} (S, X ^\ast ) $.
\end{lemma}
\begin{theorem} \label{Fatou general}
	Let $ ^\ast X$ be reflexive or $ X $ be separable.
	Let $1< p <t<r<\infty  $  or $1<p \le t <r =\infty$.
	If a bounded sequence $\{f_\ell \}_{\ell\in \mathbb N} \subset \H_{p'}^{t',r'}( ^\ast X) \cap L_{\operatorname{loc}} ^{p'}(^\ast X)$ converges locally to $f$  in the weak topology of $  L^{p'}(^\ast X) $, then $f \in \H_{p'}^{t',r'}( ^\ast X)$ and 
	\begin{equation*}
		\|f\|_{ \H_{p'}^{t',r'}( ^\ast X) }  \le \liminf _{\ell \to \infty} \|f_\ell\|_{ \H_{p'}^{t',r'}( ^\ast X) } .
	\end{equation*}
\end{theorem}

\begin{proof}
	We use the idea from \cite[Lemma 5.4]{HNSH23}. 
	Since if $ 1<p = t <r =\infty$,  $\mathcal{H}_{p'}^{t',r'} (^\ast X) = L^{p'} (^\ast X)$. We only prove the case  $1< p<t<r<\infty$ or $1< p < t<r=\infty$.
	
	By normalization, we assume $\|f_\ell\|_{ \H_{p'}^{t',r'}( ^\ast X) } \le 1$ for all $\ell \in \mathbb N$. It suffices to show  $f \in \H_{p'}^{t',r'}( ^\ast X)$ and   $\|f\|_{ \H_{p'}^{t',r'}( ^\ast X) } \le 1$.
	We write 
	\begin{equation*}
		f_\ell = \sum_{Q \in \D} \lambda_{\ell,Q} b_{\ell,Q}
	\end{equation*}
	where $b_{\ell,Q} $ is a $ (p',t', {^\ast X} ) $-block supported on $ Q $ and  $ \left( \sum_{Q\in \D} | \lambda_{\ell,Q} |^{r'} \right)^{1/r'} \le  \| f_\ell \|_{  \H_{p'}^{t',r'}( ^\ast X) }  +\epsilon \le 1+\epsilon $.
	We may suppose that $  \lambda_{\ell,Q} \ge 0$ since $ - b_{\ell,Q} $ is also a $ (p',t', {^\ast X} ) $-block supported on $ Q $.	
	Using Lemma \ref{weak* compct}, by passing to a subsequence, we assume that 
	\begin{equation*}
		\lambda_Q := \lim_{\ell \to \infty} \lambda_{\ell,Q}
	\end{equation*}
	exists in $[0,\infty)$ and that
	\begin{equation*}
		b_Q := \lim_{\ell \to \infty} b_{\ell,Q}
	\end{equation*}
	exists in the weak topology of $ L^{p'} (^\ast X)$.
	It is easy to see that supp $b_Q \subset Q$. By the Fatou lemma, 
	\begin{align*}
		\| b_Q \|_{L^{p'} (^\ast X) } & =\sup_{ \|h\|_{L^p (X) } \le 1 } \left| \int_\rn \langle b_Q (x), h(x) \rangle \d x \right|  \\
		& \le \liminf_{ \ell \to \infty } \sup_{ \|h\|_{L^p (X) } \le 1 } \left| \int_\rn \langle b_{\ell,Q} (x), h(x) \rangle \d x \right|  \\
		& \le |Q|^{1/t-1/p} .
	\end{align*}
	Hence $b_Q$ is a $ (p',t', {^\ast X} ) $-block supported on $ Q $. 
	Let 
	\begin{equation*}
		g : = \sum_{Q\in \D} \lambda_Q b_Q.
	\end{equation*}
	We claim that $f=g$. Once this is achieved, we obtain $ f \in  \H_{p'}^{t',r'}( ^\ast X)$ and
	\begin{equation} \label{f liminf 1+ eps}
		\|f\|_{ \H_{p'}^{t',r'}( ^\ast X) } \le \left( \sum_{Q\in \D} |\lambda_Q|^{r'} \right)^{1/r'} \le \liminf_{\ell \to \infty}   \left( \sum_{Q\in \D} | \lambda_{\ell,Q} |^{r'} \right)^{1/r'} \le 1+\epsilon.
	\end{equation}
	By the Lebesgue differentiation theorem (see \cite[Theorem 2.3.4]{HVVW16} for Banach space valued Lebesgue differentiation theorem), this amounts to the proof of the equality:
	\begin{equation*}
		\int_Q f (x)\d x = \int_Q g (x)\d x
	\end{equation*}
	for all $Q \in \D$. Now fix $Q_0 \in \D$. Let $ j_{Q_0} := -\log_2 ( \ell (Q_0)) $. Then $ Q_0 \in \D _{ j_{Q_0}}$.
	By the definition of $ f_{\ell} $ and the fact that the sum defining $ f_{\ell} $ converges almost everywhere on $\rn$, we have
	\begin{align*}
		\int_{Q_0} f_{\ell} (x) \d x = \sum_{Q \in \D} \lambda_{\ell,Q} \int_{Q_0}   b_{\ell,Q} (x) \d x  = \sum_{ m \in \mathbb Z } \sum_{Q \in \D_m} \lambda_{\ell,Q} \int_{Q_0}   b_{\ell,Q} (x) \d x.
	\end{align*}
	Note that 
	\begin{align*}
		\left\|  \int_{Q_0}  b_{\ell,Q} (x) \d x \right\|_{^\ast X } &  \le \int_{Q_0}  \| b_{\ell,Q} (x)\|_{^\ast X  } \d x  \\
		& \le  \left(   \int_{Q_0}  \|  b_{\ell,Q} (x)\|_{^\ast X  } ^{p'} \chi_Q (x) \d x \right)^{1/p'} |Q \cap Q_{0} |^{1/p} \\
		& \le |Q|^{1/t-1/p}  |Q \cap Q_{0} |^{1/p} .
	\end{align*}
	Since  $1/t-1/p <0$  and $ 1- r/t < 0$,
	\begin{align*}
		& \left(  \sum_{m = -\infty} ^\infty \sum_{Q \in \D_m, Q \cap  Q_{0} \neq \emptyset } |Q|^{r/t-r/p}  |Q \cap Q_{0} |^{r/p} \right)^{1/r} \\
		& =\left(  \sum_{m = -\infty} ^{ j_{Q_0} } 2^{ - r m n (1/t-1/p)} 2^{ - r j_{Q_0} n/p }  +\sum_{ m =j_{Q_0}  +1 } ^\infty  2^{ (m - j_{Q_0}) n} 2^{ -m n r /t}  \right)^{1/r} \\
		&= \left(\sum_{m= -\infty} ^{ j_{Q_0} } 2^{ -m r n (1/t-1/p)} 2^{ - r j_{Q_0} n/p }  +\sum_{ m =j_{Q_0}  +1 } ^\infty   2^{ - j_{Q_0} n} 2^{m n  (1 -r/t) } \right)^{1/r} \\
		& \le  C <\infty.
	\end{align*}
	Thus \begin{align*}
		\left\|  \int_{Q_0} f_{m_v} (x) \d x \right\|_{^\ast X} \le C \left( \sum_{m = -\infty} ^\infty \sum_{Q \in \D _m , Q \cap Q_0  \neq \emptyset }  ( \lambda_{\ell,Q}  )^{r'} \right)^{1/r'} \le  C (1+\epsilon) <\infty.
	\end{align*}
	Now we can  use  Lebesgue's convergence theorem to  obtain
	\begin{align*}
		\lim_{\ell \to \infty} \int_{Q_0}  f_{\ell} (x) \d x  
		= \sum_{m= -\infty} ^\infty \left( \lim_{\ell \to \infty} \sum_{Q \in \D _m , Q \cap Q_0  \neq \emptyset } \lambda_{\ell,Q}  \int_{Q_0} b_{\ell,Q} (x) \d x    \right).
	\end{align*}
	Since 
	$
	\sum_{Q \in \D _m , Q \cap Q_0  \neq \emptyset }
	$
	is the symbol of summation over a finite set for each $m$, we have
	\begin{equation*}
		\lim_{\ell \to \infty} \sum_{Q \in \D} \lambda_{\ell,Q}  \int_{Q_0}  b_{\ell,Q} (x) \d x =  \sum_{Q \in \D} \lambda_{Q}   \int_{Q_0}  b_{Q}  (x) \d x = \int_{Q_0} g (x)\d x.
	\end{equation*}
	By  $\{f_\ell \}_{\ell\in \mathbb N} \subset \H_{p'}^{t',r'}( ^\ast X) \cap L_{\operatorname{loc}} ^{p'}(^\ast X)$ converges locally to $f$  in the weak topology of $  L^{p'}(^\ast X) $,
	we obtain
	\begin{equation*}
		\int_{Q_0} f (x)\d x = \lim_{\ell\to \infty } \int_{Q_0} f_\ell (x)\d x = \lim_{\ell \to \infty} \sum_{Q \in \D} \lambda_{\ell,Q}  \int_{Q_0}  b_{\ell,Q} (x) \d x  = \int_{Q_0} g (x)\d x.
	\end{equation*}
	
	Finally, letting $\epsilon \to 0^+$ in  (\ref{f liminf 1+ eps}), the proof is complete.
\end{proof}

The following lemma indicates how to generate blocks.
\begin{lemma} \label{generate block}
	Let $X$ be a Banach space  such that $  ( ^\ast X )^ \ast = X $. Let $^\ast X$ be reflexive or $X$ be separable.
	Let $1<  p<t<r<\infty$ 	or $1<  p\le t<r=\infty$. Let $f$ be an $L^{p'} (^\ast X)$ function supported	on a cube $Q \in \D$. Then $\| f\|_ { \H_{p'}^{t',r'}(^\ast X) }  \le \| f\| _{ L^{p'} (^\ast X) }  |Q|^{1/p -1/t} $.
\end{lemma}
\begin{proof}
	Assume that $f$ is not zero almost everywhere, let 
	\begin{equation*}
		b := \frac{|Q|^{1/t -1/p}}{ \|f\|_{L^{p'}  (^\ast X)  }  } f.
	\end{equation*}
	Then $b$ is supported on $Q \in \D$, and 
	\begin{equation*}
		\| b \|_{L^{p'}  (^\ast X)} =       |Q|^{ 1/t - 1/p }     .  
	\end{equation*}
	Hence $b$ is a $ (p',t',{^\ast X})$-block and $\|b\|_{  \H_{p'}^{t',r'}(^\ast X) }  \le  1$. 
	Equating this inequality, we obtain the desired result.
\end{proof}

Now we use the Fatou property of $\H_{p'}^{t',r'}(^\ast X)$ to show the following result.
\begin{corollary}
	Let $X$ be a Banach space  such that $  ( ^\ast X )^ \ast = X $. Let $^\ast X$ be reflexive or $X$ be separable.
	Let $1 <  p<t<r<\infty$ 	or $1 < p\le t<r=\infty$. Suppose that we have  $ \langle  g, f  \rangle \in L^1 $ for all $f  \in  M_p^{t,r} (X)$. Then $g \in \H_{p'}^{t',r'}(^\ast X) $.
\end{corollary}
\begin{proof}
	By the closed graph theorem, $ \|  \langle f,  g \rangle  \|_{L^1}   \le L \|f\|_{ M_p^{t,r} (X) }  $  for some $L$ independent of $f  \in  M_p^{t,r}  (X) $. Set   $g_N := \chi_{B(0,N)}  \chi_{ [0,N] } (\|g\|_{^\ast X}) g $ for each $N \in \mathbb N$. Then using Lemma \ref{generate block} and the fact that $B(0,N)$ can be covered by at most $2^n$ dyadic cubes, we have $g_N \in \H_{p'}^{t',r'}(^\ast X)$. Using Theorem \ref{predual seq BM}, we obtain $ \| g_N \| _{  \H_{p'}^{t',r'}(^\ast X) }  \le L. $
	Since $g_N$  converges locally to $g$  in the weak topology of $  L^{p'}(^\ast X) $,
	using Theorem \ref{Fatou general}, we see that $g \in  \H_{p'}^{t',r'}(^\ast X) $.
\end{proof}

\subsection{Lattice property} \label{lattice space}
In \cite[Proposition 2.6]{H20}, Ho obtained the lattice property of the scalar-valued block spaces.
Since there are many spaces $X$ satisfying that $^\ast X$ is reflexive or $ X$ be separable, there is no unique criterion to define  which is bigger  when $x,y \in  {^\ast X} $. For example, when $X = \ell^q, q\in (1,\infty)$. First, we can say $y\in \ell^{q'}$ is not less than $x\in\ell^{q'}$ if $ \|x\|_{\ell^{q'}} \le \|y\|_{\ell^{q'}}  $. But, we can also define $y\in \ell^{q'}$ is not less than $x\in\ell^{q'}$ if for each $i \in \mathbb N$, $ | f_i| \le |g_i| $.  The first way can make us compare all elements in $\ell^{q'}$ but the second  be more precise.
Since we do not know the structure of $X$. We use the first way to describe the lattice property of  block spaces $\H_{p'}^{t',r'}(^\ast X) $. 

We do not show the  that block spaces $\H_{p'}^{t',r'}(^\ast X) $ have the lattice property.
So a natural question is the following:
\begin{problem}
	Let $X$ be a Banach space  such that $  ( ^\ast X )^ \ast = X $. Let $^\ast X$ be reflexive or $X$ be separable.
	Let $1< p<t<r<\infty$ 	or $1< p\le t<r=\infty$. 
	Let $g \in \H_{p'}^{t',r'}(^\ast X) $. If a  $^\ast X$-valued function $f$ satisfying $ \|f(x) \|_{^\ast X} \le \| g(x) \|_{^\ast X} $ a.e., can we obtain $f$  belongs to $\H_{p'}^{t',r'}(^\ast X) $?
\end{problem}

But if $X  = \ell^q, q \in (1,\infty)$, we have the following result.
\begin{lemma}
	Let $X  = \ell^q, q \in (1,\infty)$.
	Let $1< p<t<r<\infty$ 	or $1< p\le t<r=\infty$. 
	Then a vector valued function $\vec f$  belongs to $\H_{p'}^{t',r'}(\ell^{q'}) $ if and only if there exists a vector valued function  $\vec g \in \H_{p'}^{t',r'}(\ell^{q'})$ such that for each $i \in \mathbb N$, $|f_i(x)| \le g_i (x)$ a.e.
\end{lemma}

\begin{proof}
	Let $\vec f \in  \H_{p'}^{t',r'}(\ell^{q'})$. Then there  exist a sequence $\{ \lambda_{j,k} \}_{ (j,k) \in \mathbb Z ^{1+n} }  \in \ell ^{r'} $  and a sequence  $(p',t',\ell^{q'})$-block $\{ \vec  b_{j,k} \}  _{ (j,k) \in \mathbb Z ^{1+n} } $  supported on  $Q_{j,k}   $
	such that $\vec f 	=\sum_{(j,k)\in\mathbb{Z}^{n+1}}\lambda_{j,k} \vec  b_{j,k} $  and $\| \lambda_{j,k} \|_{\ell ^{r'}}  \le \| \vec  f\|_{\H_{p'}^{t',r'}(\ell^{q'})  }    +\epsilon$.
	For each $ (j,k) \in \mathbb Z ^{1+n} $, put 
	\begin{equation*}
		\widetilde{\vec  b_{j,k}} = \{ 	\widetilde{  b_{j,k,i}}   \}_{i=1}^\infty := \{ 	| b_{j,k,i}|   \}_{i=1}^\infty
	\end{equation*}
	and 
	\begin{equation*}
		\vec 	g = \sum_{(j,k)\in\mathbb{Z}^{n+1}} |\lambda_{j,k}|	\widetilde{\vec  b_{j,k}}.
	\end{equation*}
	Then for each $i \in \mathbb N$,
	\begin{equation*}
		|f_i| = \left| \sum_{(j,k)\in\mathbb{Z}^{n+1}} \lambda_{j,k} b_{j,k,i} \right| \le  \sum_{(j,k)\in\mathbb{Z}^{n+1}}| b_{j,k,i}| |\lambda_{j,k}|	= g_i
	\end{equation*}
	and $\|\vec g\|_{\H_{p'}^{t',r'}(\ell^{q'}) } \le \| \lambda_{j,k} \|_{\ell ^{r'}} <\infty $.	
	
	Conversely, suppose that there exists $\vec g\in \H_{p'}^{t',r'}(\ell^{q'}) $ that satisfies $|f_i(x) | \le g_i(x)$ for each $i \in \mathbb N$ a.e. Decompose $\vec  g $ as $ \vec g = \sum_{(j,k)\in\mathbb{Z}^{n+1}}\lambda_{j,k} \vec  b_{j,k}$ such that $ \| \{  \lambda_{j,k} \} \|_{\ell ^{r'}}  \le \|\vec  g \|_{ \H_{p'}^{t',r'}(\ell^{q'})  }     +\epsilon $ where $\{ \lambda_{j,k} \}_{ (j,k) \in \mathbb Z ^{1+n} }  \in \ell ^{r'} $  and $\{  \vec  b_{j,k} \}  _{ (j,k) \in \mathbb Z ^{1+n} } $ is a sequence  $(p',t', \ell^{q'})$-block. 
	For each $i \in \mathbb N$, let $E_i := \{  y\in \rn : g_i (y) > 0 \}. $ By definition, if $x\in \rn \backslash E_i$, $f_i (x) = g_i (x) =0 .$  		
	Then
	\begin{equation*}
		\chi_{E_i } (x) =  \sum_{(j,k)\in\mathbb{Z}^{n+1}}\lambda_{j,k} \frac{1}{g_i(x)} b_{j,k,i} \chi_{E _i} (x),
	\end{equation*}
	and 
	\begin{equation*}
		f_i (x) =  \sum_{(j,k)\in\mathbb{Z}^{n+1}}\lambda_{j,k} \frac{f_i(x)}{g_i(x)} b_{j,k,i} \chi_{E_i } (x).
	\end{equation*}
	We define $0/0 =0$.
	Hence
	\begin{equation*}
		\vec f (x)= \sum_{(j,k)\in\mathbb{Z}^{n+1}}\lambda_{j,k} \left\{ \frac{f_i(x)}{g_i(x)} b_{j,k,i}  \right\}_{i=1}^\infty. 
	\end{equation*}
	
	Since $ |f_i (x) | / g_i (x) \le 1 $(note that we define $0/0=0$), the vector valued function $ \left\{ \frac{f_i(x)}{g_i(x)} b_{j,k,i}  \right\}_{i=1}^\infty $ is a $(p',t', \ell^{q'})$-block supported on $Q_{j,k}$. Thus 
	\begin{equation*}
		\|\vec  f\|_ {\H_{p'}^{t',r'}(\ell^{q'})  } \le \| \{  \lambda_{j,k} \} \|_{\ell ^{r'}}  \le  \| \vec g \|_{ \H_{p'}^{t',r'}(\ell^{q'})  }     +\epsilon <\infty.
	\end{equation*}
	Taking the infimum over all admissible sequences  $\lambda$, we also have
	\begin{equation*}
		\|\vec  f\|_ {\H_{p'}^{t',r'}(\ell^{q'})  } \le  \|\vec  g\|_ {\H_{p'}^{t',r'}(\ell^{q'})  }.
	\end{equation*}
	Thus the proof is complete.
\end{proof}

\subsection{Duality of  $M_p^{t,r} (X)$ when $X$ is reflexive and $1<p<t<r<\infty$}
Note that if $X$ is reflexive, $ ^\ast X = X ^\ast$.
Before showing the duality of  $M_p^{t,r} (X)$, we need some preparation. Recall that 
for each $k \in \mathbb Z$, for a  $X$-valued measurable function $f$, the  operator $E_k$  is defined by
\begin{equation*}
	E_k (f) (x) = \sum_{Q \in \D_k} \frac{1}{|Q|} \int_Q f (y) \d y \chi_Q (x), \quad  x\in \rn.
\end{equation*}
\begin{lemma} \label{E_k opera}
	Let $X$ be a Banach space. Let $1\le p <t <r <\infty $. Then for all $f \in M_p^{t,r} (X)$, for each $k\in \mathbb Z$,
	\begin{equation*}
		\| E_k (f ) \|_{ M_p^{t,r} (X)} \le C \|f  \|_{ M_p^{t,r} (X)}
	\end{equation*}
	where the positive constant $C$ depends only on $n,r,t$.
\end{lemma}

\begin{proof}
	We write out the norm $	\| E_k (f ) \|_{ M_p^{t,r} (X)}$ fully:
	\begin{align*}
	&	\| E_k (f ) \|_{ M_p^{t,r} (X)} ^r \\
	& = \left( \sum_{j=-\infty}^k +  \sum_{j=k+1}^\infty  \right) \sum_{R \in \D_j} \left(  |R|^{1/t-1/p} \left(   \int_R   \left\| \frac{1}{|Q|} \int_Q f (y) \d y \right\|_X^p \chi_Q (x) \d x \right)^{1/p} \right)^r \\
		& =: I +II.
	\end{align*}
	We first estimate $I$. By general Minkowski's inequality and H\"older's inequality, we obtain
	\begin{align*}
		I &= \sum_{j=-\infty}^k  \sum_{R \in \D_j} \left(  |R|^{1/t-1/p} \left(  \sum_{Q \in \D_k, Q \subset R}     \left\| \frac{1}{|Q|} \int_Q f (y) \d y \right\|_X^p |Q| \right)^{1/p} \right) ^r  \\
		& \le  \sum_{j=-\infty}^k  \sum_{R \in \D_j} \left(  |R|^{1/t-1/p} \left(  \sum_{Q \in \D_k, Q \subset R}    \frac{1}{|Q|}  \int_Q \| f (y)\|_X^p \d y  |Q| \right)^{1/p} \right) ^r  \\
		& =  \sum_{j=-\infty}^k  \sum_{R \in \D_j} \left(  |R|^{1/t-1/p} \left(   \int_R \| f (y)\|_X^p \d y   \right)^{1/p} \right)^r  \\
		& \le \|f\|_{M_p^{t,r} (X) } ^r.
	\end{align*}
	When $j \ge k+1$, denote by $R_{j-k} \in \D_k$  the $(j-k)$-th dyadic parent of $R \in \D_j$. In $II$,
	\begin{equation*}
		II =  \sum_{j=k+1}^\infty \sum_{R \in \D_j} \left(  |R|^{1/t}   \left\| \frac{1}{|R_{j-k}|} \int_{R_{j-k}} f (y) \d y \right\|_X \right)^r .
	\end{equation*}
	Note that given $S \in \D_k$, there are $2^{(j-k)n}$ cubes $R\in \D_j$  such that $R_{j-k} =S$. So we have
	\begin{align*}
		II & =  \sum_{j=k+1}^\infty 2^{nr (j-k) (1/r-1/t) }  \sum_{S \in \D _k} \left(  |S|^{1/t}   \left\| \frac{1}{|S|} \int_{S} f (y) \d y \right\|_X \right)^r \\
		& \le \sum_{j=k+1}^\infty 2^{nr (j-k) (1/r-1/t) }  \sum_{S \in \D _k} \left(  |S|^{1/t}   \left( \frac{1}{|S|} \int_{S} \| f (y)\|_X^p \d y \right)^{1/p}  \right)^r \\
		& \le \|f\|_{M_p^{t,r} (X) }^r \sum_{j=k+1}^\infty 2^{nr (j-k) (1/r-1/t) } \\
		& \le C_{n,r,t}  \|f\|_{M_p^{t,r} (X) }^r .
	\end{align*}
	Combining the estimates of $I$ and $II$, we obtain 
	\begin{equation*}
		\| E_k (f ) \|_{ M_p^{t,r} (X)} \le   C_{n,r,t} \|f  \|_{ M_p^{t,r} (X)}.
	\end{equation*}
	Thus the proof is complete.
\end{proof}

\begin{lemma}[Theorem 3.18, \cite{B11}]  \label{sub converge weak topo}
	Assume that $E$ is a reflexive Banach space and let $\{ x_n\}_{n\in \mathbb N}$ be a bounded
	sequence in $E$. Then there exists a subsequence $\{ x_{n_k} \}_{ k \in \mathbb N}$ that converges in the weak
	topology $\sigma (E, E^\ast)$.
\end{lemma}
Now we are ready  to prove the dual of  $  M_p^{t,r} (X) $ is $ \H_{p'}^{t',r'}(  X ^\ast)  $.
\begin{theorem} \label{dual BMX}
	Let $X$ be reflexive. Let $1<p<t<r<\infty$. Then the dual of $  M_p^{t,r} (X) $ is $ \H_{p'}^{t',r'}(  X ^\ast)  $.
\end{theorem}
\begin{proof}
	Let $L \in  (  M_p^{t,r} (X) )^\ast$. Then for any $\ell \in \mathbb N$ and $k \in \mathbb Z$, the mapping $f \in L^p (X)  \mapsto L (   E_k ( \chi_{B_\ell} f )  )  \in \mathbb R $  is a bounded linear mapping. 
	Indeed, for any $f  \in L^p (X) $, from the proof of Lemma \ref{E_k opera}, using $1<p<t<r<\infty$, we have
	\begin{align*}
		& \| E_k (f) \|_{ M_p^{t,r} (X) } ^r \\
		 & \le \sum_{j=-\infty}^k  \sum_{R \in \D_j} \left(  |R|^{1/t-1/p} \left(   \int_R \| f (y)\|_X^p \d y   \right)^{1/p} \right)^r  \\
		&	\quad + \sum_{j=k+1}^\infty 2^{nr (j-k) (1/r-1/t) }  \sum_{S \in \D _k} \left(  |S|^{1/t}   \left( \frac{1}{|S|} \int_{S} \| f (y)\|_X^p \d y \right)^{1/p}  \right)^r \\
		& \le \sum_{j=-\infty}^k  2^{ -jn (r/t-r/p)}  \left( \sum_{R \in \D_j}    \int_R \| f (y)\|_X^p \d y   \right)^{r/p}  \\
		& \quad + \sum_{j=k+1}^\infty 2^{nr (j-k) (1/r-1/t) }   2^{-knr (1/t-1/p)}     \left( \sum_{S \in \D _k} \int_{S} \| f (y)\|_X^p \d y \right)^{r/p}   \\
		& \lesssim \|f\|_{L^p (X)} ^r,
	\end{align*}
	where the implicit positive constant depends on $k$. Hence, any $\ell \in \mathbb N$, $k \in \mathbb Z$ and any $f  \in L^p (X) $,
	\begin{align} \label{L bound Lp}
		\nonumber
		|  L (   E_k ( \chi_{B_\ell} f )  ) | & \le \|L\|_{ M_p^{t,r} (X) \to \mathbb R } \| E_k ( \chi_{B_\ell} f ) \|_{  M_p^{t,r} (X) }  \\
		& \lesssim  \|L\|_{ M_p^{t,r} (X) \to \mathbb R } \|  \chi_{B_\ell} f  \|_{ L^p (X) }.
	\end{align}
	Thus the mapping $f \in L^p (X)  \mapsto L (   E_k ( \chi_{B_\ell} f )  )  \in \mathbb R $  is a bounded linear mapping. 
	
	By the duality $L^p (X)   -L^{p'} (X^\ast) $, there exists $g_{k,\ell} \in L^{p'} (X^\ast)$  such that
	\begin{equation} \label{L k l f =}
		L ( E_k ( \chi_{B_\ell } f)   )  =\int_\rn \langle f(x), g_{k,\ell} (x) \rangle \d x
	\end{equation}
	for all $f \in L^{p} (X)$. Since 
	\begin{align*}
		\int_\rn \langle f(x), g_{k,\ell} (x) \rangle \d x & = 	L ( E_k ( \chi_{B_\ell f})   ) = 	L ( E_k (  \chi_{B_{\ell+1}} \chi_{B_\ell f})   )  \\
		& = \int_\rn \chi_{B_\ell} (x) \langle f(x), g_{k,\ell+1} (x) \rangle \d x
	\end{align*}
	for all $f \in L^{p} (X)$, we have $ g_{k,\ell}  = g_{k,\ell+1} \chi_{B_\ell} $ almost everywhere. Therefore, the limit $g_k := \lim_{\ell \to \infty} g_{k,\ell}$ exists almost everywhere and belongs to $L_{\operatorname{loc}}^{p'} (X^\ast) $.
	
	For any  $\ell \in \mathbb N$, $k \in \mathbb Z$, there exist finitely many dyadic cubes $\{ Q_j\}_{j=1}^{2^n}  \subset \D _{\nu}$ for some $\nu  \in \mathbb Z$ such that supp $g_{k,\ell} \subset\cup_{j=1}^{2^n}  Q_j  $. From this, it is  easily to see  that, for $m\in \mathbb N$, $\ell \in \mathbb N$, $k \in \mathbb Z$, $\chi_{[0,m]} ( \| g_{k,\ell} \|_{X^\ast}) g_{k,\ell} \in  \H_{p'}^{t',r'}(  X ^\ast) $. Now we prove for any  $\ell \in \mathbb N$, $k \in \mathbb Z$,
	$\{ \chi_{[0,m]} ( \| g_{k,\ell} \|_{X^\ast}) g_{k,\ell} \}_{m \in \mathbb N}$
	is a bounded  sequence in $\H_{p'}^{t',r'}(  X ^\ast)$. Indeed, by Theorem \ref{predual seq BM}, we have
	\begin{align*}
		& \|\chi_{[0,m]} ( \| g_{k,\ell} \|_{X^\ast} ) g_{k,\ell} \|_{\H_{p'}^{t',r'}(  X ^\ast) } \\
		& = \max_{ \|h\|_{ M_p^{t,r}  (X) }  \le  1  } \int_\rn  | \langle  \| \chi_{[0,m]} ( \| g_{k,\ell} \|_{X^\ast} ) g_{k,\ell} , h (x)  \rangle | \d x  \\
		& \le  \max_{ \|h\|_{ M_p^{t,r}  (X) }  \le  1  } \int_\rn  | \langle   g_{k,\ell} , h (x)  \rangle| \d x  \\
		& 	=  \sup_{ \|h\|_{ M_p^{t,r}  (X) }  \le  1  } L ( E_k (\chi_{B_\ell}   \sgn ( \langle g_{k,\ell}, h   \rangle ) h )  )  \\
		& \le  \sup_{ \|h\|_{ M_p^{t,r}  (X) }  \le  1  }   \|L\|_{ M_p^{t,r} (X) \to \mathbb R }  \|E_k \|_{  M_p^{t,r} (X) \to  M_p^{t,r} (X) } \|  \sgn ( \langle g_{k,\ell}, h   \rangle ) h  \|_{M_p^{t,r} (X) } \\
		& \lesssim  \|L\|_{ M_p^{t,r} (X) \to \mathbb R }  ,
	\end{align*}
	where the implicit positive constant is independent of both $k$ and $\ell$. Thus,   $\{ \chi_{[0,m]} ( \| g_{k,\ell} \|_{X^\ast}) g_{k,\ell} \}_{m \in \mathbb N}$
	is a bounded  sequence in $\H_{p'}^{t',r'}(  X ^\ast)$  for any  $\ell \in \mathbb N$, $k \in \mathbb Z$. 
	Moreover, from (\ref{L bound Lp}), (\ref{L k l f =}), and the Lebesgue dominated convergence theorem, it follows that $\chi_{[0,m]} ( \| g_{k,\ell} \|_{X^\ast}) g_{k,\ell} $  converges locally to $g_{k,\ell}$ in the weak topology of $L^{p'} (X^\ast)$ as $m \to \infty$. Thus, by Lemma \ref{Fatou general}, we obtain $g_{k,\ell} \in \H_{p'}^{t',r'}(  X ^\ast) $ and for any $\ell \in \mathbb N$, $k \in \mathbb Z$, 
	\begin{equation*}
		\| g_{k,\ell} \|_{ \H_{p'}^{t',r'}(  X ^\ast) }  \lesssim  \|L\|_{ M_p^{t,r} (X) \to \mathbb R }
	\end{equation*}
	where the implicit positive constant is independent of both $k$ and $\ell$.
	
	Notice that, for any given $k\in \mathbb Z$  and any given compact set $F \subset \rn$, there exists $\tilde \ell \in \mathbb N$, depending on both $k$ and $F$, such that
	\begin{equation*}
		\chi_F	g_k  = \chi_F g_{k,\tilde \ell}.
	\end{equation*}
	From this, (\ref{L bound Lp}), (\ref{L k l f =}), and the Lebesgue dominated convergence theorem, we obtain, for any $k \in \mathbb Z$, $  g_{k,\ell}$ converges locally to $g_k$  in the weak topology of $L^{p'} (X^\ast)$ as $\ell \to \infty$. Applying this and Lemma \ref{Fatou general} again, we  obtain, for any $k \in \mathbb Z$,
	\begin{equation*}
		\| g_{k} \|_{ \H_{p'}^{t',r'}(  X ^\ast) }  \lesssim  \|L\|_{ M_p^{t,r} (X) \to \mathbb R }
	\end{equation*}
	where the implicit positive constant is independent of  $k$.
	Since $g_k \in \H_{p'}^{t',r'}(  X ^\ast) $, for any $k \in \mathbb Z$, and almost everywhere $x\in \rn$,
	\begin{equation*}
		g_k (x) =\sum_{Q \in \D} \lambda_{k,Q} b_{k,Q}(x)
	\end{equation*}
	where $\|  \{\lambda_{k,Q} \}_{Q \in \D}\|_{ \ell^{r'}  } \le (1+\epsilon) \| g_k \|_{ \H_{p'}^{t',r'}(  X ^\ast)} $
	and $b_{k,Q}$ is a $ (p', t', X^\ast)$ block supported on $Q$. 
	It is easy to see that for any  $k \in \mathbb Z$, $Q\in \D$, $| \lambda_{k,Q}| \le  (1+\epsilon) \| g_k \|_{ \H_{p'}^{t',r'}(  X ^\ast)}$. We also may assume that $  \lambda_{k,Q} \ge 0$ since $ - b_{\ell,Q} $ is also a $ (p',t', {^\ast X} ) $-block supported on $ Q $.
	
	By this and Lemma \ref{sub converge weak topo}, there exists a subsequence of $ \{	g_k \}_{k\in \mathbb Z}$ (we still denote it by $ \{	g_k \}_{k\in \mathbb Z}$) such that for any $Q \in \D$, 
	\begin{equation*}
		\lambda_Q := \lim_{k \to \infty}  \lambda_{k,Q}
	\end{equation*}
	exits in $[0,\infty)$ and 
	\begin{equation*}
		b_Q :=  \lim_{k \to \infty}  b_{k,Q}
	\end{equation*}
	exists in the weak topology of $L^{p'} (X^\ast)$. Moreover, for each $Q\in \D$,
	\begin{align*}
		\| b_Q \|_{L^{p'} (X^\ast) } & =\sup_{ \|h\|_{L^p (X) } \le 1 } \left| \int_\rn \langle b_Q (x), h(x) \rangle \d x \right|  \\
		& \le \lim_{ k \to \infty } \sup_{ \|h\|_{L^p (X) } \le 1 } \left| \int_\rn \langle b_{k,Q} (x), h(x) \rangle \d x \right|  \\
		& \le |Q|^{1/t-1/p}.
	\end{align*}
	Hence $b_Q$ is a $ (p',t', {^\ast X} ) $-block supported on $ Q $. Applying the Fatou Lemma,
	we have
	\begin{equation*}
		\|  \{\lambda_{Q} \}_{Q \in \D}\|_{ \ell^{r'}  } \le \liminf_{k\to \infty}	\|  \{\lambda_{k,Q} \}_{Q \in \D}\|_{ \ell^{r'}  } \le (1+\epsilon) \| g_k \|_{ \H_{p'}^{t',r'}(  X ^\ast)}  \lesssim  \|L\|_{ M_p^{t,r} (X) \to \mathbb R } .
	\end{equation*}
	By this, and Proposition \ref{block sum converge}, we have $g : = \sum_{Q\in \D}\lambda_{Q} b_{Q}(x) \in  \H_{p'}^{t',r'}(  X ^\ast)  $. We have proved in Theorem \ref{Fatou general}, for any $Q\in \D$,
	\begin{equation} \label{int g_k =int g}
		\lim_{k\to \infty} \int_Q g_k (x) \d x = \int_Q g (x) \d x.
	\end{equation}
	
	Now we show that there exists a unique $g \in \H_{p'}^{t',r'}(  X ^\ast)$ such that, for any $h \in M_p^{t,r} (X) $, 
	\begin{equation*}
		L (h) = \int_\rn \langle g (x), h (x) \rangle \d x .
	\end{equation*} 
	Since $L_c^\infty (X)$ is dense in $M_p^{t,r} (X) $ when $1<p<t<r<\infty$, we may assume that $h \in L_c^\infty (X)$.
	
	From (\ref{L k l f =}),  the proven conclusion that $g_k = \lim_{\ell \to \infty} g_{k,\ell}$ exists almost everywhere, and the Lebesgue dominated convergence theorem, we deduce that, for any $h \in L_c^\infty (X) \subset M_p^{t,r} (X) $, any $k \in \mathbb N$,
	\begin{align} \label{L e_K h}
		L (E_k (h) )  = \lim_{\ell \to \infty } 	L (E_k (h  \chi_{B_\ell} ) )  = \lim_{\ell \to \infty } \int_\rn \langle h (x) , g_{k,\ell} (x) \rangle \d x  = \int_\rn \langle h (x) , g_{k} (x) \rangle \d x 
	\end{align}  
	where the first equality holds true because $h $ has a compact support. From Theorem \ref{E_k f to f}, we have
	\begin{equation*}
		\lim_{k\to \infty} \| E_k (h) -h \|_{M_p^{t,r} (X) } =0.
	\end{equation*}
	Applying this, the continuity of $ L,$ (\ref{L e_K h}) and (\ref{int g_k =int g}), we obtain
	\begin{equation*}
		L (h) =  \lim_{k\to \infty} L (E_k (h) ) = \lim_{k\to \infty} \int_\rn \langle h (x) , g_{k} (x) \rangle \d x = \int_\rn \langle h (x) , g (x) \rangle \d x
	\end{equation*}
	for all $h \in L_c^\infty (X)  $. Since $L_c^\infty (X)$  is dense in $M_p^{t,r} (X) $  and $g \in \H_{p'}^{t',r'}(  X ^\ast)$, we still have 
	\begin{equation*}
		L (h) =   \int_\rn \langle h (x) , g (x) \rangle \d x
	\end{equation*}
	for all $ h \in M_p^{t,r} (X)$. Thus, $L$  is realized by $g \in \H_{p'}^{t',r'}(  X ^\ast)$ and the proof is complete.
\end{proof}

\begin{corollary}\label{reflexive BMX}
	Let $X$ be reflexive. Let $1<p<t<r<\infty$. Then  $  M_p^{t,r} (X) $ is reflexive.
\end{corollary}
\begin{remark}
	Let $X$ be $\mathbb R$, then Theorem \ref{dual BMX} becomes \cite[Theorem 5.5]{HNSH23} and Corollary \ref{reflexive BMX}  becomes \cite[Remark 5.7]{HNSH23}.
\end{remark}

\section{Hardy-Littlewood maximal function on block spaces}\label{HL block}
The Hardy-Littlewood maximal operator is important in harmonic analysis. Therefore, in this section, we study the boundedness of maximal average operator on $X$-valued Bourgain-Morrey spaces and $^\ast X$-valued block spaces. Since we do not know the accurate form of $X$, it is difficult to define the precise maximal average operator. However, the vector-valued Bourgain-Morrey is important because it is useful in establishing the theory of related function spaces.

The classical definition of the maximal operator immediately extends to the context of $X$-valued functions simply by replacing the absolute values by the norms:
\begin{equation*}
	\M _X (f) (x) = \sup_{Q \ni x}\frac{1}{|Q|} \int_Q \| f(y) \|_X \d x.
\end{equation*}
Since $\M_X (f) = \M ( \|f\|_X )$, from \cite[Lemma 4.1]{HNSH23} (or by \cite[Corollary 4.7]{ZSTYY23}, or by \cite[Theorem 4.5]{HLY23}, or by \cite[Theorem 5.5]{ZYZ24}), we immediately obtain the following result.
\begin{lemma}
	Let $1<  p<  t < r <\infty $ or $1<  p\le t < r =\infty$. Then $\M_X$ is a bounded  operator from $M_p^{t,r} (X)$ to $M_p^{t,r} $.
\end{lemma}

\begin{lemma} \label{norm equ}
	Let $0< p\le t \le r \le \infty $.  Then $f \in M_{p}^{t,r} (X)$ if and only if $f \in L_{\mathrm{loc}}^{p} (X)$ and
	\begin{equation*}
		[f]_{ M_{p}^{t,r} (X) } := \left(  \int_0^\infty \int_\rn \left( |B(y,t)| ^{1/t -1/p -1/r}  \left\| \|f\|_{X} \right\|_{L^p (B(y,t)  )}  \right) ^r \d y \frac{\d t}{t}   \right)^{1/r} <\infty,
	\end{equation*}
	with the usual modifications made when $r =\infty$. Moreover,
	\begin{equation*}
		\|f \|_{ M_{p}^{t,r} (X)}  \approx [f]_{ M_{p}^{t,r} (X) } .
	\end{equation*}
\end{lemma}
\begin{proof}
	Repeating  \cite[Theorem 2.9]{ZSTYY23} and replacing  $ | \cdot |$ by $\| \cdot\|_{X}$, we obtain the result. 
\end{proof}

\subsection{$X = \ell^q, q \in (1,\infty)$}
Note that for a $X$-valued function $\vec f$,  the vector Hardy-Littlewood maximal operator $\M \vec f$ is defined by $	\M \vec f := \{ \M f_i \}_{i=1}^\infty.$
The main result of this subsection is the following:
\begin{theorem} \label{M eta r < infty = infty}
	Let $1<q < \infty $.  Let $1<p<t<r<\infty$ or let $1<p \le t <r =\infty$.  Let $ 0<\eta < \min\{p' , q' \}$. Then $\M _\eta$ is bounded on $\mathcal{H}_{p'}^{t',r'} (\ell^{q'})$.
\end{theorem}

In \cite[Theorem 4.3]{ST09}, Sawano and Tanaka proved the Hardy-Littlewood  maximal operator $\M$ on vector valued on block spaces $ \mathcal{H}_{p'}^{t',1} (\ell^{q'})$; see also \cite[Theorem 2.12]{IST15}. Their method also works for $\M _\eta $ with $1< \eta < \min\{ p' ,q' \}$. Hence, we only need to prove the case $1<p<t<r<\infty$.  We use the idea from \cite[Section 6.2 Besov-Bourgain-Morrey spaces]{ZYY242}. Before proving  Theorem \ref{M eta r < infty = infty}, we give some preparation.

\begin{remark} 
	Let $1<q\le \infty $. Let $1<p<t<r<\infty$ or $  1 <p \le t <r =\infty$. From \cite[Theorem 4.3]{HNSH23}, we know that  the vector  Hardy-Littlewood maximal  operator $\M$ is bounded on  $ M_{p}^{t,r} (\ell^q)$. By Lemma \ref{norm equ}, we obtain
	\begin{equation*}
		[ \M  \vec f]_{ M_{p}^{t,r} (\ell^q) } \lesssim [ \vec f]_{ M_{p}^{t,r} (\ell^q) }.
	\end{equation*}		
\end{remark}
\begin{lemma}\label{HL each j}
	Let $1<q\le \infty $. Let $1<p<t<r<\infty$ or $  1 <p \le t <r =\infty$. Then for each $v \in \mathbb Z $,
	\begin{align*}
		& \left(  \sum_{m \in \mathbb Z^n} \left(  |Q_{v,m}|^{1/t-1/p}\Big(\int_{Q_{v,m}} \left( \sum_{i=1}^\infty \M f_i (x) ^q \right)^{p/q}\d x\Big)^{1/p}  \right) ^r \right) ^{1/r} \\
		& \lesssim	\left(  \sum_{m \in \mathbb Z^n} \left(  |Q_{v,m}|^{1/t-1/p}\Big(\int_{Q_{v,m}} \left( \sum_{i=1}^\infty |f_i (x)| ^q \right)^{p/q}\d x\Big)^{1/p}  \right) ^r \right) ^{1/r} .
	\end{align*}
\end{lemma}
\begin{proof}
	Repeating  the proof of  \cite[Theorem 4.3]{HNSH23}, the result is obtained and we omit the detail here.
\end{proof}
\begin{remark}
	Let $1<q\le \infty $. Let $1<p<t<r<\infty$ or $  1 <p \le t <r =\infty$.
	By Lemmas \ref{norm equ} and \ref{HL each j}, we have that for each $j\in\mathbb Z$,
	\begin{align*}
		& 	\left( \int_\rn \left( |B(y,2^{-j})| ^{1/t -1/p -1/r} \int_{B(y, 2^{-j})  } \left(  \sum_{i=1}^\infty \M f_i (x) ^q \right)^{p/q} \d x  \right) ^{r} \d y \right) ^{1/r}  \\
		& \lesssim 	\left( \int_\rn \left( |B(y,2^{-j})| ^{1/t -1/p -1/r} \int_{B(y, 2^{-j})  } \left(  \sum_{i=1}^\infty | f_i (x)| ^q \right)^{p/q} \d x  \right) ^{r} \d y \right) ^{1/r}.
	\end{align*}
\end{remark}

\begin{definition}\label{def vec skice}
	Let $1<q<\infty$.
	Let $p,t,r \in (1,\infty) $ and $j \in \mathbb Z$. The slice space $  (\mathcal E_{p'}^{t',r'}   )_j  (\ell^{q'}) $ is define to be the set of all  $\vec f \in L_{\operatorname{loc} } ^{p'} (\ell^{q'})$  such that 
	\begin{equation*}
		\|\vec  f\|_{  (\mathcal E_{p'}^{t',r'}   )_j  (\ell^{q'} ) } := \left( \sum_{k\in \mathbb Z^n } \left( |Q_{j,k}|^{1/t' -1/p' }   \| \chi_{ Q_{j,k} }\vec  f\|_{ L^{p'} (\ell^{q'}) }   \right) ^{r'}  \right) ^{1/r'} <\infty.
	\end{equation*}
\end{definition}

In \cite[Proposition 8.3]{AP17}, Auscher and Prisuelos-Arribas showed the boundedness of Hardy-Littlewood on scalar slice space $(\mathcal E_{p'}^{t',r'}   )_j$. 
From Lemma \ref{HL each j},  we know that $\M$ is bounded on   $  (\mathcal E_{p'}^{t',r'}   )_j  (\ell^{q'}) $.

In the following lemma, we establish a characterization of vector valued slice space $(\mathcal E_{p'}^{t',r'}   )_j (\ell^{q'})$.

\begin{lemma} \label{char vec slice}
	Let $1<q< \infty $.  Let $1<p<t<r<\infty$.  
	Let $j \in \mathbb Z$. Then $\vec f \in (\mathcal E_{p'}^{t',r'}   )_j (\ell^{q'}) $ if and only if there exist a sequence $  \{ \lambda_{j,k}\}_{k\in \mathbb Z^n}  \subset \mathbb R$ satisfying
	\begin{equation*}
		\left(  \sum_{m\in \mathbb Z^n} |\lambda_{j,k}|^{r'} \right)^{1/r'} <\infty
	\end{equation*}
	and a sequence $\{\vec  b_{j,k} \}_{k\in \mathbb Z^n}$ of vector functions on $\rn$ satisfying, for each $k\in \mathbb Z^n$, supp $\vec b_{j,k}\subset Q_{j,k}$  and 
	\begin{equation} \label{b jk = |Q|}
		\left\|  \|b_{j,k}\|_{\ell^{q'}}  \right\|_{ L^{p'} (Q_{j,k}) }  = |Q_{j,k}|^{1/p' - 1/t'}
	\end{equation}
	such that $\vec f = \sum_{k \in \mathbb Z^n} \lambda_{j,k} \vec  b_{j,k} $ almost everywhere on $\rn$; moreover, for such $\vec f$,
	\begin{equation*}
		\| \vec  f\|_{ (\mathcal E_{p'}^{t',r'}   )_j (\ell^{q'}) }  = 	\left(  \sum_{m\in \mathbb Z^n} |\lambda_{j,k}|^{r'} \right)^{1/r'}  .
	\end{equation*}
\end{lemma}
\begin{proof}
	We first show the necessity. Let $\vec f \in (\mathcal E_{p'}^{t',r'}   )_j (\ell^{q'}) $. For each $m \in \mathbb Z^n$, when $  \left\|  \|\vec f\|_{\ell^{q'}}  \right\|_{ L^{p'} (Q_{j,k}) } >0 $, let $ \lambda_{j,k} := |Q_{j,k}|^{1/t' -1/p' }   \| \chi_{ Q_{j,k} } \vec f\|_{ L^{p'} (\ell^{q'}) }  $,   $\vec  b_{j,k} := \lambda_{j,k} ^{-1} \vec f \chi_{Q_{j,k} }$ and, when $  \left\|  \|\vec f\|_{\ell^{q'}}  \right\|_{ L^{p'} (Q_{j,k}) } = 0 $, let $ \lambda_{j,k} := 0$ and $\vec  b_{j,k} : = |Q_{j,k}|^{ -1/t'} \chi_{ Q_{j,k} }  \vec a $ where $\vec a = (1, 0, 0, \ldots) \in \ell^{q'} $ with $\|\vec a\|_{ \ell^{q'}} =1$. Then by the Definition \ref{def vec skice}, it is easy to show that supp $\vec b_{j,k}\subset Q_{j,k}$, (\ref{b jk = |Q|}) and 
	\begin{equation*}
		\left(  \sum_{m\in \mathbb Z^n} |\lambda_{j,k}|^{r'} \right)^{1/r'}  = 	\|\vec f\|_{ (\mathcal E_{p'}^{t',r'}   )_j (\ell^{q'}) }  	  <\infty .
	\end{equation*} 	
	Next we show the sufficiency. Assume that there exist a sequence $  \{ \lambda_{j,k}\}_{k\in \mathbb Z^n}  \in \ell^{r'}$ and a sequence $\{\vec  b_{j,k} \}_{k\in \mathbb Z^n}$ satisfying supp $\vec b_{j,k} \subset Q_{j,k}$ and  (\ref{b jk = |Q|}) such that $\vec f =  \sum_{k \in \mathbb Z^n} \lambda_{j,k} \vec b_{j,k}$  almost everywhere on $\rn$. Note that $ \{ Q_{j,k}\}_{k\in \mathbb Z^n}$  are disjoint. Then 
	\begin{align*}
		\|\vec  f\|_{ (\mathcal E_{p'}^{t',r'}   )_j (\ell^{q'}) }   & = \left( \sum_{k\in \mathbb Z^n } \left( |Q_{j,k}|^{1/t' -1/p' }   \| \lambda_{j,k}\vec  b_{j,k}\|_{ L^{p'} (\ell^{q'}) }   \right) ^{r'}  \right) ^{1/r'}  \\
		& = 	\left(  \sum_{m\in \mathbb Z^n} |\lambda_{j,k}|^{r'} \right)^{1/r'}  <\infty.
	\end{align*}
	Hence $\vec f \in (\mathcal E_{p'}^{t',r'}   )_j (\ell^{q'})$. Thus we finish the proof.
\end{proof}
Using the characterization of vector valued slice space $(\mathcal E_{p'}^{t',r'}   )_j (\ell^{q'})$, we obtain the following characterization of block spaces.
\begin{lemma}\label{char block}
	Let $1<q < \infty $.  Let $1<p<t<r<\infty$.  Then $\vec f \in \mathcal{H}_{p'}^{t',r'} (\ell^{q'})$ if and only if 
	\begin{align*}
		[\vec f]_{\mathcal{H}_{p'}^{t',r'} (\ell^{q'})} & : = \inf\left\{ \left( \sum_{j\in \mathbb Z} \|\vec  f_j\|_{ (\mathcal E_{p'}^{t',r'}   )_j (\ell^{q'}) } ^{r'} \right) ^{1/r' }  :  \vec f = \sum_{j\in \mathbb Z} \vec f_j, \vec f_j \in (\mathcal E_{p'}^{t',r'}   )_j (\ell^{q'})  \right\} \\
		& <\infty ;
	\end{align*}
	moreover, for such $\vec f$, $ \|\vec  f\|_{ \mathcal{H}_{p'}^{t',r'} (\ell^{q'})} =  [\vec f]_{\mathcal{H}_{p'}^{t',r'} (\ell^{q'})}$.
\end{lemma}
\begin{proof}
	We first show the necessity. Let $\vec f \in \mathcal{H}_{p'}^{t',r'} (\ell^{q'})$. Then	there exist a sequence $\lambda=\{\lambda_{j,k}\}_{(j,k)\in\mathbb{Z}^{n+1}}\in\ell^{r'}(\mathbb{Z}^{n+1})$ and a sequence $\{\vec  b_{j,k}  \}_{(j,k)\in\mathbb{Z}^{n+1} }  $  of  $(p',t', \ell^{q'})$-block such that 
	\begin{equation*}
		\vec f=\sum_{(j,k)\in\mathbb{Z}^{n+1}}\lambda_{j,k} \vec b_{j,k}  
	\end{equation*}
	almost everywhere on $\rn$ and $ \left( \sum_{ (j,k)\in\mathbb{Z}^{n+1} } |\lambda_{j,k}|^{r'}\right)^{1/r'}  < (1+\epsilon) \|\vec  f\|_{ \mathcal{H}_{p'}^{t',r'} (\ell^{q'})}$.
	For each $j \in \mathbb Z$, let $\vec f_j : = \sum_{k \in \mathbb Z^n}  \lambda_{j,k} \vec  b_{j,k}  $. Then we obtain
	\begin{align*}
	&	\left(  \sum_{j\in \mathbb Z}	\| \vec f_j \|_{ (\mathcal E_{p'}^{t',r'}   )_j (\ell^{q'}) }  ^{r'}  \right)^{1/r'} \\
	 & = \left(  \sum_{j\in \mathbb Z} \left( \sum_{k\in \mathbb Z^n } \left( |Q_{j,k}|^{1/t' -1/p' }   \| \lambda_{j,k} \vec b_{j,k}\|_{ L^{p'} (\ell^{q'}) }   \right) ^{r'}  \right)^{r'/r'}   \right) ^{1/r'} \\
		& \le 	\left(  \sum_{j\in \mathbb Z} \sum_{m\in \mathbb Z^n} |\lambda_{j,k}|^{r'} \right)^{1/r'} < (1+\epsilon) \| \vec  f\|_{ \mathcal{H}_{p'}^{t',r'} (\ell^{q'})}.
	\end{align*}
	Letting $\epsilon \to 0^+$, we obtain $  [\vec f]_{\mathcal{H}_{p'}^{t',r'} (\ell^{q'})} \le  \|\vec  f\|_{ \mathcal{H}_{p'}^{t',r'} (\ell^{q'})} $.
	
	Next we show the sufficiency. Let $\vec f \in \mathscr M (vec)$ with $ [\vec f]_{\mathcal{H}_{p'}^{t',r'} (\ell^{q'})} <\infty$. Then there exists a sequence $\{\vec f_j \}_{j\in \mathbb Z}$ satisfying that $\vec f_j \in  (\mathcal E_{p'}^{t',r'}   )_j (\ell^{q'})  $ for each $j$ such that $\vec f =\sum_{j\in \mathbb Z} \vec f_j $ almost everywhere on $\rn$ and 
	\begin{equation*}
		\left( \sum_{j\in \mathbb Z} \|\vec  f_j\|_{ (\mathcal E_{p'}^{t',r'}   )_j (\ell^{q'}) } ^{r'} \right) ^{1/r' } < (1+\epsilon) [\vec f]_{\mathcal{H}_{p'}^{t',r'} (\ell^{q'})}.
	\end{equation*}
	By  Lemma \ref{char vec slice}, for each $j\in \mathbb Z$, there exist a sequence 
	$\{ \lambda_{j,k} \}_{ (j,k)\in\mathbb{Z}^{n+1} }  \in \ell^{r'}$  and 
	a sequence $ \{\vec  b_{j,k}  \}_{ (j,k) \in \mathbb{Z}^{n+1} }$  of vector valued satisfying (\ref{b jk = |Q|}) such that $\vec f_j = \sum_{k \in \mathbb Z^n} \lambda_{j,k}\vec  b_{j,k} $ almost everywhere on $\rn$ and 
	\begin{equation*}
		\|\vec  f_j \|_{ (\mathcal E_{p'}^{t',r'}   )_j (\ell^{q'}) }  = 	\left(  \sum_{m\in \mathbb Z^n} |\lambda_{j,k}|^{r'} \right)^{1/r'} .
	\end{equation*}
	Hence, $\vec f = \sum_{j\in \mathbb Z}\vec f_j = \sum_{j\in \mathbb Z}\sum_{k \in \mathbb Z^n} \lambda_{j,k}\vec  b_{j,k}  $ almost everywhere on $\rn$ and 
	\begin{equation*}
		\|\vec  f\|_{ \mathcal{H}_{p'}^{t',r'} (\ell^{q'}) }  \le 	\left(  \sum_{j\in \mathbb Z} \sum_{m\in \mathbb Z^n} |\lambda_{j,k}|^{r'} \right)^{1/r'} < (1+\epsilon) [\vec f]_{\mathcal{H}_{p'}^{t',r'} (\ell^{q'})}.
	\end{equation*}
	Letting $\epsilon \to 0^+$, we obtain $	\| \vec f\|_{ \mathcal{H}_{p'}^{t',r'} (\ell^{q'}) } \le   [\vec f]_{\mathcal{H}_{p'}^{t',r'} (\ell^{q'})} $. This finish the proof of sufficiency. Thus we complete the proof of Lemma \ref{char block}.
\end{proof}
Now we are ready to show  Theorem \ref{M eta r < infty = infty}.
\begin{proof}[Proof of Theorem \ref{M eta r < infty = infty}]
	We may suppose that $ 1\le \eta < \min\{ p' ,q' \}  $ since $\M_{\eta_0}\vec  f \le  \M_{\eta_1} \vec f $ when $ 0< \eta_0 \le \eta_1$. Thus $\M_\eta$  is a sublinear operator.
	Let $0<\epsilon <1$.
	Let $\vec f\in \mathcal{H}_{p'}^{t',r'} (\ell^{q'})$. From Lemma \ref{char block}, there exists a sequence $\{\vec f_j\}_{j\in \mathbb Z} \subset (\mathcal E_{p'}^{t',r'}   )_j (\ell^{q'}) $ such that $\vec f = \sum_{j\in \mathbb Z} \vec f_j$ almost everywhere on $\rn$ and 
	\begin{equation*}
		\| \vec f\|_{\mathcal{H}_{p'}^{t',r'} (\ell^{q'}) } = [\vec f]_{\mathcal{H}_{p'}^{t',r'} (\ell^{q'})} > (1 - \epsilon) \left( \sum_{j\in \mathbb Z} \| \vec f_j\|_{ (\mathcal E_{p'}^{t',r'}   )_j (\ell^{q'}) } ^{r'} \right) ^{1/r' } .
	\end{equation*}
	By the sublinear of Hardy-Littlewood maximal function and Lemma \ref{HL each j}, we obtain
	\begin{align*}
		\| \M_\eta \vec f\|_{\mathcal{H}_{p'}^{t',r'} (\ell^{q'})}  & \le \left\| \sum_{j\in   \mathbb Z}  \M_\eta \vec f _j  \right\|_{ \mathcal{H}_{p'}^{t',r'} (\ell^{q'}) }  = \left[ \sum_{j\in   \mathbb Z}  \M_\eta \vec f _j \right]_{ \mathcal{H}_{p'}^{t',r'} (\ell^{q'})   }  \\
		& \le \left( \sum_{j\in   \mathbb Z} \| \M_\eta \vec f _j  \|_{ (\mathcal E_{p'}^{t',r'}   )_j (\ell^{q'}) } ^{r'} \right)^{1/r'}  \lesssim  \left( \sum_{j\in   \mathbb Z} \|  \vec f _j  \|_{ (\mathcal E_{p'}^{t',r'}   )_j (\ell^{q'}) } ^{r'} \right)^{1/r'}  \\
		& < \frac{1}{1-\epsilon} 	\| \vec f\|_{\mathcal{H}_{p'}^{t',r'} (\ell^{q'}) }
	\end{align*}
	where the implicit positive constants are independent of both $\vec f$ and $\epsilon$. Letting $\epsilon \to 0^+$, we obtain $	\| \M_\eta \vec  f\|_{\mathcal{H}_{p'}^{t',r'} (\ell^{q'})} \lesssim  	\| \vec f\|_{\mathcal{H}_{p'}^{t',r'} (\ell^{q'}) } $ and hence complete the proof.
\end{proof}

\subsection*{Acknowledgment}
The corresponding author
Jingshi Xu is supported by the National Natural Science Foundation of China (Grant No. 12161022) and the Science and Technology Project of Guangxi (Guike AD23023002).
Pengfei Guo is supported by Hainan Provincial Natural Science Foundation of China (Grant No. 122RC652).

\end{document}